\documentclass{amsproc}

\usepackage{amsmath}
\usepackage{amssymb}
\usepackage[mathcal]{eucal}
\usepackage[bb=pazo]{mathalfa}
\usepackage{rotating}
\usepackage{graphicx}
\usepackage[tableposition=below]{caption}
\usepackage{pigpen}
\usepackage{multirow}
\usepackage{array}
\newcolumntype{H}{>{\setbox0=\hbox\bgroup}c<{\egroup}@{}}

\DeclareGraphicsExtensions{.png,.pdf,.eps}
\usepackage[all,2cell,cmtip]{xy}
\usepackage{tikz}
\usepackage{color}
\usepackage{bm}
\usepackage{longtable}
\usepackage{soul}
\newtheorem{theorem}{Theorem}[section]
\newtheorem{lemma}[theorem]{Lemma}
\newtheorem{corollary}[theorem]{Corollary}
\newtheorem{proposition}[theorem]{Proposition}
\newtheorem{conjecture}[theorem]{Conjecture}

\newtheorem{notation}[theorem]{Notation}

\theoremstyle{definition}

\newtheorem{definition}[theorem]{Definition}

\newtheorem{problem}{Problem}
\newtheorem{remark}[theorem]{Remark}

\newtheorem{example}[theorem]{Example}

\def\C{{\mathbb C}}

\def\G{{\mathbb G}}

\def\P{{\mathbb P}}
\def\Q{{\mathbb Q}}
\def\R{{\mathbb R}}
\def\Z{{\mathbb Z}}

\def\cD{{\mathcal D}}
\def\cE{{\mathcal E}}

\def\cI{{\mathcal I}}

\def\cK{{\mathcal K}}

\def\cM{{\mathcal M}}

\def\cO{{\mathcal{O}}}

\def\cQ{{\mathcal{Q}}}

\def\cS{{\mathcal S}}

\def\cU{{\mathcal U}}

\def\Q{{\mathbb{Q}}}
\def\G{{\mathbb{G}}}

\def\fh{{\mathfrak h}}

\def\PGL{\operatorname{PGl}}

\def\Mo{\operatorname{\hspace{0cm}M}}
\def\Na{\operatorname{\hspace{0cm}L}}

\def\DA{{\rm A}}
\def\DB{{\rm B}}
\def\DC{{\rm C}}
\def\DD{{\rm D}}
\def\DE{{\rm E}}
\def\DF{{\rm F}}
\def\DG{{\rm G}}

\def\lra{\longrightarrow}

\def\LRa{\hspace{-0.13cm}\Leftrightarrow\hspace{-0.13cm}}

\def\lra{\longrightarrow}

\def\operatorname#1{\mathop{\rm #1}\nolimits}

\def\Pic{\operatorname{Pic}}

\def\rk{\operatorname{rk}}

\def\deg{\operatorname{deg}}

\def\det{\operatorname{det}}

\def\qed{\hspace{\fill}$\rule{2mm}{2mm}$}

\def\r{{\operatorname{r}}}
\def\gd{{\operatorname{g.d.}}}
\def\ed{{\operatorname{e.d.}}}
\def\Cox{{\operatorname{\hspace{0cm}h}}}

\def\NU{{\operatorname{N^1}}}
\def\HH{\operatorname{H\hspace{0.5pt}}}

\def\ngb{\operatorname{ngb}}
\def\dis{\operatorname{dis}}

\def\Coeff{\operatorname{Coeff}}

\newcommand{\ol}[1]{\overline{#1}}

\newcommand{\pb}{\ar@{}[dr]|(.50){\text{\pigpenfont J}}}

\makeatletter
 \newcommand*\tl[1]{\mathpalette\wthelper{#1}}
\newcommand*\wthelper[2]{%
        \hbox{\dimen@\accentfontxheight#1%
                \accentfontxheight#11.15\dimen@
                $\m@th#1\widetilde{#2}$%
                \accentfontxheight#1\dimen@
        }%
}

\newcommand*\accentfontxheight[1]{%
        \fontdimen5\ifx#1\displaystyle
                \textfont
        \else\ifx#1\textstyle
                \textfont
        \else\ifx#1\scriptstyle
                \scriptfont
        \else
                \scriptscriptfont
        \fi\fi\fi3
}
\makeatother

\makeindex

\begin{document}

\title[Splitting conjectures for uniform flag bundles]{Splitting conjectures for uniform flag bundles\\ {\normalsize\rm (\today)}}

\author[R. Mu\~noz]{Roberto Mu\~noz}
\address{Departamento de Matem\'atica Aplicada, ESCET, Universidad
Rey Juan Carlos, 28933-M\'ostoles, Madrid, Spain}
\email{roberto.munoz@urjc.es}

\author[G. Occhetta]{Gianluca Occhetta}
\address{Dipartimento di Matematica, Universit\`a di Trento, via
Sommarive 14 I-38123 Povo di Trento (TN), Italy} \thanks{First and third author supported by the Spanish government project MTM2015-65968-P, second and third author supported by PRIN project 2015EYPTSB\_002:  ``Geometria delle variet\`a algebriche''.}
\email{gianluca.occhetta@unitn.it}

\author[L.E. Sol\'a Conde]{Luis E. Sol\'a Conde}
\address{Dipartimento di Matematica, Universit\`a di Trento, via
Sommarive 14 I-38123 Povo di Trento (TN), Italy}
\email{lesolac@gmail.com}


\begin{abstract}
We present here some conjectures on the diagonalizability of uniform principal bundles on rational homogeneous spaces, that are natural extensions of classical theorems on uniform vector bundles on the projective space, and study the validity of these conjectures in the case of classical groups.
\end{abstract}

\maketitle

\section{Introduction}\label{sec:intro}

Within the framework of the theory of vector bundles, the concept of uniformity and the problem of explaining its relation with homogeneity arose in the early 1960's, and it is usually attributed to R.L.E. Schwarzenberger, and H. Grauert (see \cite[I,~Section~3]{OSS}, and the references therein). Building upon Grothendieck's theorem on the complete reducibility of vector bundles on the Riemann sphere $\P^1$, it seems natural to start the study of vector bundles on $\P^n$ by focusing on those whose restriction to every line is the same and the simplest examples of the kind are the homogeneous bundles. Obviously, direct sums of line bundles --the so-called {\it diagonalizable} bundles-- are uniform and the main question is to construct non diagonalizable ones. Thanks to \cite{EHS,Sa} we know that the only examples of uniform non diagonalizable bundles of rank smaller than or equal to $n$ on $\P^n$ are constructed upon $T_{\P^n}$ (by twisting and/or dualizing) and so they are homogeneous. Examples of non homogeneous uniform bundles on $\P^n$ are known to exist in higher rank (\cite{Drez80,Elen}, \cite[I,~Theorem~3.3.2]{OSS}).

Similar results have been proved for other particular examples of rational homogeneous manifolds (\cite{Ballico2,Ballico1, Ellia,Guyot, MOS22,MOS1,VdV,Wis2}).  In the archetypal case of Grassmannians,  non diagonalizable uniform vector bundles of the lowest rank are constructed upon universal bundles. In other words, they are constructed upon a principal bundle subjacent to the family of lines on the variety --with respect to which the uniformity is defined--.   

More generally, the natural extension of the concept of ``universal bundle'' to the case of a general rational homogeneous space is not anymore a vector bundle but a $G$-principal bundle 
(for a certain semisimple group $G$), 
upon which one constructs homogeneous vector bundles by means of representations of $G$.
Hence the relation between homogeneity and uniformity should be explained, within the general framework of rational homogeneous spaces and representation theory, in the following terms:

\begin{problem} {\it Classify low rank uniform principal $G$-bundles ($G$ semisimple algebraic group) on rational homogeneous spaces.}
\end{problem}

Among all the possible projective realizations of $G$-principal bundles, we have chosen to work here, as in the preceding article \cite{MOS5}, with the associated $G/B$-bundles (where $B$ is a Borel subgroup of $G$). The reasons are that, on one hand, the $G$-bundle can be reconstructed upon the associated $G/B$-bundle and, on the other, flag varieties $G/B$ have a particularly well known geometry, determined by their rational curves of minimal degree and the combinatorics of the root system associated with $G$. In this paper we will make use of the cohomology ring of $G/B$, that was completely described by Borel in terms of the root system and the Weyl group of $G$ (see Theorem \ref{thm:BGG}). 

In the above setting, one expects the universal bundles of a given rational homogeneous manifold $X$ to provide a certain bound under which the only uniform flag bundles on $X$ are diagonalizable (see Definition \ref{def:diagonalizable}). A natural way of expressing this is to define the rank $\r(\pi)$ of a $G/B$-bundle $\pi:Y\to X$ as the semisimple rank of the group $G$ and $\r(X)$ to be the minimum of $\r$ on the universal bundles on $X$. However, computations suggest that one should use a more subtle way of comparing flag bundles on $X$ with its universal bundles; we then define $\Cox(\pi)$ to be the Coxeter number of the Dynkin diagram of $G$ (see Section \ref{ssec:flagbundles}, and Remark \ref{rem:coxeterandfundamental}), and $\Cox(X)$ to be the minimum of $\Cox$ on the universal bundles on $X$. In either case, examples support the following:

\begin{conjecture}[Splitting conjectures]\label{conj:main}
Let $X$ be a rational homogenous manifold of Picard number one, quotient of a simple algebraic  group $\ol G$ and $\cM$ its family of $\ol G$-isotropic lines. Let $G$ be another simple algebraic group, and let $\pi: Y \to X$ be a $G/B$-bundle, uniform with respect to $\cM$. 
\begin{itemize}
\item[$(\Cox)$]  If $\Cox(X)>\Cox(\pi)$ then $\pi: Y \to X$ is diagonalizable.
\item[$(\r)$]   If $\r(X)>\r(\pi)$ then $\pi: Y \to X$ is diagonalizable.
\end{itemize}
\end{conjecture}

In this paper we will show that the Splitting conjecture $(\Cox)$ holds in the case in which $\ol{G}$ and $G$ are of classical type (Theorem \ref{thm:main}), and that the Splitting conjecture $(\r)$ holds when $G$ is of type $\DA_n$ (Proposition \ref{prop:evidenceAk}). 

The structure of the paper is the following. In Section \ref{sec:prelim}, after recalling some basic facts on flag bundles and uniformity, we show how their diagonalizability can be written in terms of the constancy of certain maps to rational homogeneous spaces. We then describe the cohomology rings of flags and rational homogeneous manifolds, with particular attention to the classical cases  in Section \ref{sec:cohomflags}; we use this description in Section \ref{sec:morRH} to prove some cohomological restrictions on morphisms to rational homogeneous spaces (Proposition \ref{prop:classical}); summing up, we prove Theorem \ref{thm:main} in Section \ref{sec:main}. The last section is devoted to discussing the Splitting conjecture $(\r)$.

\section{Setup and preliminaries}\label{sec:prelim}

\subsection{Flag bundles}\label{ssec:flagbundles}
Along this section, $X$ will denote a complex projective algebraic variety and $G$ a semisimple Lie group. We will fix a Borel subgroup $B\subset G$. 
A {\it $G/B$-bundle on} $X$ is, by definition, a smooth morphism $\pi:Y\to X$ whose fibers are isomorphic to $G/B$. We will always assume that $X$ is simply connected (which is the case if, for instance, $X$ is Fano), and then, given $\pi$, we may choose (see \cite[Remark 2.1]{OSW}) $G$ to be the identity component of the automorphism group of $G/B$, so that $\pi$ is determined by a cocycle $\theta\in \HH^1(X,G)$. Moreover, $\pi:Y\to X$ may be constructed upon the {\it $G$-principal bundle $\pi_G:E\to X$ associated to }$\theta$, since we may identify $Y$ with the algebraic variety $$E\times^GG/B:=(E\times G/B)/\sim,\qquad (e,gB)\sim(eh,h^{-1}gB),\,\,\forall h\in G.$$ Under this identification, $\pi$ corresponds to the natural map sending the class of $(e,gB)$ to $\pi_G(e)$. 

In a similar way, given any finite dimensional complex linear representation $V$ of $G$, the variety $E\times^GV$  defines a vector bundle over $X$, whose fibers are isomorphic to $V$. Moreover, a projective representation $\P(V)$ of $G$ defines a projective bundle over $X$.

If we consider a maximal torus $H\subset B$, it determines a root system $\Phi$, whose Weyl group $W$ is isomorphic to the quotient ${\rm N}(H)/H$ of the normalizer ${\rm N}(H)$ of $H$ in $G$. Within $\Phi$, $B$ provides a base of positive simple roots $\Delta=\{\alpha_i,\,\, i=1,\dots,n\}$, whose associated reflections we denote by $s_i$. Finally, as usual, we consider the Dynkin diagram $\cD$, whose set of nodes is $\Delta$. When $G$ is simple the nodes of the Dynkin diagram will be numbered as in the standard reference \cite[p.~58]{Hum3} and we will identify each node $\alpha_i\in \Delta$ with the corresponding index $i$. We will say, as usual, $G$ {\it of classical type} (resp. {\it exceptional}) when its Dynkin diagram is of type $\DA_n$, $\DB_n$, $\DC_n$ or $\DD_n$ (resp. $\DE_n$, $\DF_n$ or $\DG_n$). 
The {\em rank} of the semisimple group $G$ is defined as $\rk(G):=\dim H=\sharp(\Delta)$. The order of the composition $s_1\circ \dots\circ s_n\in W$ will be called the {\em Coxeter number} of $\cD$, see \cite[3.18]{Hum3}. 

Finally, we denote by $\NU(Y|X)$ the cokernel of the pullback map $\NU(X)\to \NU(Y)$, between the real vector spaces of classes of $\R$-divisors in $X$ and $Y$. The dimension of $\NU(Y|X)$, which is equal to $\rk(G)$, will be called the {\it rank of} $\pi$. 

\subsection{Partial flag bundles}\label{ssec:partial}
For every subset $I\subset \Delta$  
we may considered a parabolic subgroup $P(I)$ defined by $P(I):=BW(I)B$, where $W(I)\subset W$ is the subgroup of $W$ generated by the reflections $s_i$, $i\in I$. Given a $G/B$-bundle $\pi:Y\to X$ as above, for every such subset $I\subset \Delta$,  there is a factorization:
$$
\xymatrix{Y\ar@/^1pc/[rr]^{\pi}\ar[r]_{\rho_I}&Y_I\ar[r]_{\pi_I}&X}
$$
where $Y_I:=E\times^GG/P(I)$. Note that $\rho_I$ is a flag bundle over $Y_I$, with fibers $P(I)/B\cong G(I)/B(I)$, where $G(I)$ denotes the semisimple group obtained as a quotient of $P(I)$ by its radical (we simply call it the {\it semisimple quotient of} $P(I)$),  and $B(I)$ denotes a Borel subgroup of $G(I)$. 

For $I=\{i\}$, we simply write $Y_i,\rho_i$ and $\pi_i$. The map $\rho_i:Y\to Y_i$ is a $\P^1$-bundle, whose relative canonical divisor and fiber we denote by $K_i$ and $\Gamma_i$, respectively. Note that the relative canonical divisors $K_i$, $i\in\Delta$, form a basis of  
$\NU(Y|X)$, and that the matrix
$
(-K_i\cdot\Gamma_j)
$
is equal to the Cartan matrix of $G$ (cf. \cite[Proposition ~3]{MOSW}).

\subsection{Examples}\label{ssec:univbund} 
The simplest way of constructing $G/B$-bundles over $X$ is by means of a Cartan subgroup $H\subset G$:

\begin{definition}\label{def:diagonalizable}
Let $X$ be an algebraic variety, $\pi:Y\to X$ be a $G/B$-bundle over $X$ defined by a cocycle $\theta\in \HH^1(X,G)$. Then we say that the bundle $\pi$ is {\it diagonalizable} if $\theta$ lies in the image of the natural map $\HH^1(X,H)\to\HH^1(X,G)$, where $H\subset B$ is a maximal torus.
\end{definition}

Note that $\HH^1(X,H)\cong \HH^1(X,\C^*)^n\cong\Pic(X)^n$, hence the choice of $n$ line bundles on $X$ determines  a diagonalizable $G/B$-bundle over $X$. 
Diagonalizable $G/B$-bundles admit $\sharp(W)$ disjoint sections, one for each possible Borel subgroup of $G$ containing $H$. 

On the other hand, the best known non diagonalizable flag bundles are universal bundles, which are defined over rational homogeneous spaces. Let us first introduce some notation, that we will use in the sequel.

\begin{notation}{\rm Let $G$ be a semisimple group with Dynkin diagram $\cD$. It is well known that the projective quotients of $G$ are determined by the different markings of $\cD$. For this reason, we will set:
$$
\cD(\Delta\setminus I):=G/P(I),
$$
for every subset $I$ of the set of nodes $\Delta$ of $\cD$.
}

\end{notation}

\begin{example}\label{ex:univbund} 
Let $X$ be rational homogeneous, of the form $X={G}/{P(I)}=\cD(\Delta\setminus I)$. Denoting by ${B}$ the Borel subgroup of ${G}$ contained in ${P(I)}$, the map $\rho_I:{G}/{B}\to \cD(\Delta\setminus I)$ is a flag bundle, whose fibers are isomorphic to ${G}({I})/B({I})$, where ${G}({I})$ denotes the semisimple quotient of ${P}({I})$, and $B({I})\subset {G}({I})$ is a Borel subgroup. The Dynkin diagram of $G(I)$ is the subdiagram $\cD_I$ of the Dynkin diagram of ${G}$ supported on the set of nodes ${I}$. If $\cD_I$ has $k$ connected components (supported on nodes $I_s\subset I$, $s=1,\dots,k$), this ${G}({I})/B({I})$-bundle will be a fiber product over $X$ of $k$ flag bundles of the form: 
$$
\cD((\Delta\setminus I)\cup I_s)\lra X=\cD(\Delta\setminus I).
$$
Each one of these is called a {\it universal flag bundle over $X$}.

For instance, over $X=\P^m=\DA_m(1)$, the universal flag bundle is just the complete flag bundle associated to $\P(T_X)$; on the Grassmannian $\G(k,m)=\DA_m(k)$ of $k$-dimensional projective subspaces in $\P^m$, we have two universal flag bundles, $$\DA_m(\{i\leq k\})\to \DA_m(k)\quad\mbox{ and }\quad\DA_m(\{i\geq k\})\to \DA_m(k),$$ which are the complete flag bundles associated to the projectivizations of its classical universal subbundle and quotient bundle.
\end{example}

\subsection{Flag bundles on rational curves}

In the case in which $X$ is the projective line $\P^1$, Grothendieck's theorem \cite{Gro1} tells us that $G$-principal bundles depend only on some discrete data. In fact, given a Cartan subgroup $H$ of the semisimple group $G$ this theorem states that every $G$-principal bundle over $\P^1$ is diagonalizable, in the sense that the natural map
$$
\HH^1(\P^1,H)\to \HH^1(\P^1,G)
$$
is surjective. Moreover, this map corresponds precisely to the quotient of $\HH^1(\P^1,H)$ by the natural action of the associated Weyl group $W$.

Now, by considering the exponential map from $\fh$ to $H$, whose kernel is $\Na(H)$, taking into account that $\HH^1(\P^1,\fh)=0$, we have:
$$
\HH^1(\P^1,H)\cong \HH^2(\P^1,\Na(H))\cong\Na(H).
$$

One may interpret this result geometrically on the associated $G/B$-bundles (see \cite[Section 3.3]{OSW}). If $\pi:Y\to \P^1$ is the $G/B$-bundle associated to a cocycle $\theta\in \HH^1(\P^1,H)\cong\Na(H)$, we may consider a minimal section $\Gamma_0$ of $\pi$ over $\P^1$, that is a section whose deformations with a point fixed are trivial. The $n$-tuple $\delta=(d_1,\dots,d_n)$ formed by the integers obtained by intersecting $\Gamma_0$ with the relative canonical divisors $K_j$, $j=1,\dots,n$ is called the {\it tag of the flag bundle} $\pi$, and it determines it completely. 

Since every $K_j$ corresponds to a node of the diagram $\cD$, it makes sense to represent the above data by a {\it tagged Dynkin diagram}, that is, the Dynkin diagram of the Lie algebra decorated with the integer $d_i$ at the node corresponding to $\alpha_i$, for each $i$. The tagged Dynkin diagram determines completely the flag bundle over $\P^1$. 
For instance, the tag of the $\PGL_{m+1}/B$-bundle associated to the projectivization $\P(\cE)$ of a vector bundle $\cE$ on $\P^1$ is computed as the successive differences of the splitting type of $\cE$ (see \cite[Remark 4.1]{MOS5}).

\subsection{Uniformity of flag bundles}

We may now discuss the concept of uniformity of flag bundles on higher-dimensional varieties. This is an extension of a classical concept within the framework of vector bundles (cf. \cite[\S 3]{OSS}), that applies to a triple $(X,\cM,\cE)$, where $X$ is an algebraic variety, $\cM$ is a family of rational curves on $X$, and $\cE$ is a vector bundle on $X$. Then $\cE$ is said to be {\it uniform with respect to $\cM$} if the (isomorphism class of the) pullback of $\cE$ via the normalization of one of the curves of the family does not depend on the chosen curve.

Let us consider a $G/B$-bundle $\pi:Y\to X$ on a projective variety $X$, and a family of rational curves $\cM$ on $X$ (not necessarily complete). Denote by $p:\cU \to \cM$ the universal family of $\cM$ and by $q:\cU\to X$ the evaluation morphism, that we will usually assume to dominate $X$. We may consider the pullback $q^*Y:=Y\times_X\cU$, which is a $G/B$-bundle over $\cU$, whose natural morphism onto $\cU$ will be denoted by $\pi$, by abuse of notation. Then, for every rational curve $\Gamma=p^{-1}(z)\subset\cU$ we may consider the tagged Dynkin diagram of the restriction of $q^*Y$ to $\Gamma$, and pose the following definition:

\begin{definition}\label{def:redu}
Given a  projective variety $X$, a dominating family of rational curves $\cM$ on $X$, and a flag bundle $\pi:Y\to X$, we say that $Y$ is {\it uniform with respect to $\cM$} if the tag of $Y$ on $\Gamma$ is independent of the choice of the curve $\Gamma=p^{-1}(z)$, $z\in\cM$. Obviously, if $\pi$ is uniform, the vector bundles associated to $\pi$, via linear representations of the group $G$, are all uniform as vector bundles.
\end{definition}

The next example shows that universal bundles on rational homogeneous spaces $G/P(I)$ are uniform, with respect to a family of rational curves determined by the group $G$. For simplicity, we present only the case in which $G/P(I)$ has Picard number one.

\begin{example} \label{ex:unifhomog}
Consider a semisimple group ${G}$, and a maximal parabolic subgroup ${P}({I})\subset {G}$, where ${I}:=\Delta\setminus\{{\jmath}\}$, for some node ${\jmath}\in \Delta$  (notation as in Section \ref{ssec:flagbundles}). We have a decomposition:
$$
\Delta=\{{\jmath}\}\sqcup \ngb({\jmath}) \sqcup \dis({\jmath}), 
$$
where $\ngb({\jmath})$ is the set of neighboring nodes of ${\jmath}$ (meaning that they are joined to ${\jmath}$ by an edge of the Dynkin diagram). The homogeneous space $$\cM:=\cD(\ngb({\jmath}))$$ 
parametrizes a family of rational curves in $X:=\cD(\jmath)$, whose universal family and evaluation: 
$$
\xymatrix@C=15mm{\cM&\cU:=\cD(\{{\jmath}\}\sqcup\ngb({\jmath}))\ar[l]_(.70){p}\ar[r]^(.70)q&X}
$$
are the natural morphisms given by the inclusions of the corresponding parabolic subgroups. This is not, in general, a complete family of minimal rational curves in $X$ (it is properly contained in it precisely in the case in which ${\jmath}$ is an exposed short node, see \cite[Theorem 4.3]{LM}), but a closed subvariety. Nevertheless, it is a covering family of minimal rational curves, that we call the  family of {\it ${G}$-isotropic lines in $X$}, with respect to which $X$ is rationally chain connected. As usual, we set $\cM_x:=q^{-1}(x)$, for $x\in X$, which is the rational homogeneous space $\cD_I(\ngb(\jmath))$.
The complete flag manifold $\tl{\cU}:={G}/{B}$ can be considered as a ${G}({I})/{B}({I})$-bundle over $X$ (where ${B}({I})\subset {G}({I})$ denotes the Borel subgroup of ${G}({I})$ contained in ${P}(\dis({\jmath}))$), which is uniform with respect to the family $\cM$; in fact, the tag of $\tl{\cU}\to X$ on every $G$-isotropic curve is given by $d_i=-K_i\cdot\Gamma_{\jmath}$, on every node $i\in I$. 
\end{example}

\subsection{A criterion for diagonalizability of uniform flag bundles}\label{ssec:redunif}

Let $X$ be a Fano manifold of Picard number one and $\pi:Y\to X$ a $G/B$-bundle over $X$, uniform with respect to an unsplit dominating family $\cM$ of rational curves, with tag $\delta=(d_1,\dots d_n)$.
Let us set 
\begin{equation}\label{eq:I0}
I_0:=\left\{i\in \Delta|\,\,d_i=0\right\}\subset \Delta,
\end{equation}
and  denote by $P(I_0)\subset G$ the corresponding parabolic subgroup (so that the fibers of the submersion $\rho_{I_0}:Y\to Y_{I_0}$ are flag manifolds associated to the semisimple group $G({I_0})$, whose Dynkin diagram is $\cD_{I_0}$). 
Then over every rational curve $\Gamma$ of the family we have a well defined trivial proper subflag $F_{I_0}\times\Gamma\subset\pi^{-1}(\Gamma)$, where $F_{I_0}$ is a fiber of $\rho_{I_0}$. This subflag determines a section of the restriction of the pullback $q^*Y_{I_0}\to \cU$ to $\Gamma$, for every $\Gamma$, and one may then prove that all these sections glue together into a section:  
\begin{equation}\label{eq:section0}
s_{0}:\cU\to Y_{I_0}. 
\end{equation}
See \cite[Section 5.1]{MOS5} for details.

\begin{definition}\label{def:decomp2}
With the same notation as above, given any set of nodes $I\subsetneq \Delta$ containing $I_0$, we say that $\pi:Y\to X$ is {\it uniformly reducible with respect to $\cM$ and $I$},  if and only if the composition $s_I:=\rho_{I_0,I}\circ s_{0}$ (where $\rho_{I_0,I}:Y_{I_0}\to Y_I$ denotes the natural projection) factors via $q:\cU\to X$. In other words, if there exists a morphism $\sigma_{I}:X\to Y_I$ such that the following diagram is commutative:
\begin{equation}\label{diag:decomp2}
\xymatrix{&& Y_{I_0} \ar[d]^{\rho_{I_0,I}}\\\cU\ar[rr]^{s_I}\ar[rd]_q\ar[rru]^{s_{0}}&&Y_{I}\ar[dl]_{\pi_I}\\&X\ar@/_3mm/[ur]_{\sigma_I}&}
\end{equation}
Combining \cite[Corollary 5.5]{MOS5} and \cite[Lemma 5.3]{MOS5}, we can state the following:
\end{definition}

\begin{proposition}\label{prop:diagcrit}
Let $X$, $\cM$, $\pi:Y\to X$ and $I_0$ be as above. If $Y$ is uniformly reducible with respect to $\cM$ and $I= \Delta \setminus \{r\}$, for every $r\in\Delta\setminus I_0$, then $\pi$ is diagonalizable. In particular, if $\cM_x$ does not admit non constant maps to $\cD(r)$, $r\in \Delta\setminus I_0$, then $\pi$ is diagonalizable.
\end{proposition}

\section{Cohomology of rational homogeneous varieties}\label{sec:cohomflags}

In this section we recall some well known facts on the cohomology ring of rational homogeneous varieties. We will always use cohomology with real coefficients, and set $\HH^{\bullet}(M):=\HH^{\bullet}(M, \R),$ for any complex variety $M$.

\subsection{Cohomology of flag varieties}
Let $\NU(G/B)$ be the vector space of linear combinations with coefficients in $\R$ of line bundles on $G/B$ modulo numerical equivalence. It can be endowed with an action of the Weyl group $W$ of $G$, as follows: 
given a Cartan subgroup $H\subset B$, we may consider the group of characters $\Mo(H)$ of $H$; every character corresponds to a line bundle on $G/B$, and we have $\NU(G/B) \cong \Mo(H)\otimes_\Z\R$. Under this identification, the elements $\{-K_1,\dots, -K_n\}$ correspond to a base of positive simple roots $\{\alpha_1,\dots,\alpha_n\}$ of the root system $\Phi\subset \Mo(H)\subset \NU(G/B)$, and the elements of the Weyl group of $\Phi$ provide linear automorphisms of $\NU(G/B)$. Moreover,  $\NU(G/B)$ supports 
a scalar product $(\cdot,\cdot)$ which is $W$-invariant, so that the elements of $W$ are isometries with respect to it. 

The action of $W$ on $\NU(G/B)$ extends to an action of $W$ on the symmetric algebra:
$$S\NU(G/B) =\bigoplus_{k\geq 0}S^k\NU(G/B) ,\mbox{ where }S^0\NU(G/B):=\R.$$
The corresponding invariant subalgebra, denoted by $$ S\NU(G/B)^W=\{p\in S\NU(G/B) |\,\, w(p)=p\mbox{ for all }w\in W\},$$
is known to be a polynomial algebra (see, for instance, \cite[p.~77]{Hi82}), generated by $n=\rk(G)$ homogeneous elements, whose degrees, called {\em fundamental degrees of $G$} (cf. \cite[p.~87]{Hi82}), are uniquely determined by the Weyl group $W$. In the following table we present the fundamental degrees of the Weyl groups of the simple algebraic groups.

\renewcommand*{\arraystretch}{1.2}
\begin{table}[h!]
\centering
\begin{tabular}{r|l}
\hline
$\cD$&Fundamental degrees\\\hline
$\DA_n$&$2,3,\dots,n,n+1$\\
$\DB_n$,$\DC_n$&$2,4,\dots,2(n-1),2n$\\
$\DD_n$&$2,4,\dots,2(n-1),n$\\
$\DE_6$&$2,5,6,8,9,12$\\
$\DE_7$&$2,6,8,10,12,14,18$\\
$\DE_8$&$2,8,12,14,18,20,24,30$\\
$\DF_4$&$2,6,8,12$\\
$\DG_2$&$2,6$\\\hline
\end{tabular} 
\caption{Fundamental degrees of the simple algebraic groups}
\label{tab:fundeg}
\end{table}

Obviously, the set of fundamental degrees of a semisimple group is the union of the set of fundamental degrees of its simple factors. 

\begin{remark}\label{rem:coxeterandfundamental} It is well known that the maximum of the fundamental degrees of $G$ equals the Coxeter number of $\cD$.
\end{remark}

Denoting by $(S\NU(G/B)^W)_+\subset S\NU(G/B)^W$ the subset of invariant polynomials of positive degree, and by $\cI:=\big((S\NU(G/B)^W)_+\big)\subset S\NU(G/B) $ the ideal generated by them, the corresponding quotient $S\NU(G/B) /\cI$ is called the {\it ring of coinvariants of $W$} on $\NU(G/B) $, and we have the following well known result due to Borel (cf. \cite[Proposition 26.1]{Bo}):

\begin{theorem}\label{thm:BGG} With the same notation as above:
$$
\HH^{\bullet}(G/B)\cong \dfrac{S\NU(G/B) }{\cI},
$$
defined by assigning to every $\R$-divisor class $L\in \NU(G/B)$ its corresponding cohomology class $c_1(L)\in \HH^2(G/B)$.
\end{theorem}

Moreover, given a set of indices $I\subset\{1,\dots,n\}$, Theorem \ref{thm:BGG}, together with Leray--Hirsch theorem provides the following description of the cohomology ring of the rational homogeneous space $\cD(\Delta\setminus I)$:

\begin{corollary}\label{cor:BGG}
With the same notation as above:
$$
\HH^{\bullet}(\cD(\Delta\setminus I))
\cong \left(\dfrac{S\NU(G/B) }{\cI}\right)^{W(I)}\cong \dfrac{S\NU(G/B)^{W(I)} }{\cI}.$$
\end{corollary}

\subsection{Cohomology ring of rational homogeneous spaces of Picard number one: classical groups}\label{ssec:picardone}

In the case in which $G$ is of classical type, using Theorem \ref{thm:BGG} and Corollary \ref{cor:BGG} one may easily write explicit presentations of the cohomology ring of the corresponding Picard number one varieties, in terms of symmetric polynomials. 

Following \cite[p. 64]{Hum2} one can choose elements $x_1, \dots,x_{n+\epsilon}\in N^1(G/B)$ ($\epsilon=1$ when $\DA_n$ or $\epsilon=0$ when $\DB_n$, $\DC_n$ or $\DD_n$) , satisfying the relations with the roots $\alpha_j$ displayed in Table \ref{tab:bases}.

\begin{table}[h!]\centering
\begin{tabular}{|c|c|}\hline
$\cD$ &  Relation \\ \hline
& \\[-12pt] 
$\DA_n$& $\alpha_i=x_i-x_{i+1}$ \\[3pt]
$\DB_n$& $\alpha_i=x_i-x_{i+1}, \; i<n$, $\alpha_n=x_n$ \\[2pt]
$\DC_n$& $\alpha_i=x_i-x_{i+1}, \; i<n$, $\alpha_n=2x_n$ \\[2pt]
$\DD_n$& $\alpha_i=x_i-x_{i+1}, \; i<n$, $\alpha_n=x_{n-1}+x_{n}$ \\[2pt]
\hline
\end{tabular}\caption{Relation between $\alpha_i$'s and $x_j$'s}
\label{tab:bases}
\end{table}
In the cases $\DB_n,\DC_n,\DD_n$ the $x_i$'s form a basis of $N^1(G/B)$, whilst in the case $\DA_n$, they are a system of generators satisfying the linear relation $\sum_ix_i=0$.
For every type, the reflection $s_i$ corresponding to $\alpha_i$ can now be described as the transposition $(x_i,x_{i+1})$, for $i<n$.
The reflection $s_n$ is the transposition $(x_n,x_{n+1})$ in the case $\DA_n$, the change of sign of $x_n$ in the case $\DB_n$ and $\DC_n$, and the simultaneous change of sign of $x_{n-1}$ and $x_n$ in the case $\DD_n$.

Theorem \ref{thm:BGG} tells us then that we may write:
$$
\HH^\bullet(G/B)=\R[x_1, \dots,x_{n+\epsilon}]/\cI,
$$
If we set $\eta_n=x_1 \cdots x_n$ and define the polynomials 
$$Q(t) = \prod_{i=1}^{n+1}(1+tx_i), \qquad K(t) = \prod_{i=1}^n(1-t^2x_i^2),$$
whose coefficients are the elementary symmetric polynomials in $x_1, \dots, x_{n+1}$ and, up to a sign, the elementary symmetric polynomials in $x_1^2, \dots, x_{n}^2$, then we can describe generators of $\cI$ as shown in Table \ref{tab:coho} (in which $\Coeff_+(p(t))$ stands for the set of coefficients of positive degree of the polynomial $p(t)$ in the variable $t$, and $\Sigma_n$ for the symmetric group of degree $n$); note that the coefficient of maximal degree of $K(t)$ is equal to $(-1)^n\eta_n^2$.

\begin{table}[h]\centering
\begin{tabular}{|c|c|c|c|}\hline
$\cD$ &  $W$ & $\cI$ & Generators\\ 
\hline 
&&& \\[-9pt]
$\DA_n$& $\Sigma_{n+1}$& $\big(\R[x_1, \dots, x_{n+1}]^{\Sigma_{n+1}}_+\big)$ & $\Coeff_+(Q(t))$ \\[7pt] 
$\DB_n$, $\DC_n$& $\Sigma_n \ltimes \Z_2^n$ & $\big(\R[x_1^2, \dots, x_{n}^2]^{\Sigma_{n}}_+\big)$ & $\Coeff_+(K(t))$\\[7pt] 
$\DD_n$ & $\Sigma_n  \ltimes \Z_2^{n-1}$ &$\big(\R[x_1^2, \dots, x_{n}^2]^{\Sigma_{n}}_+, \eta_n=x_1 \cdots x_n\big)$& $\Coeff_+(K(t))\cup\{\eta_n\}$\\[3pt] 
\hline
\end{tabular} 
\caption{Cohomology rings of complete flags}
\label{tab:coho}
\end{table}


According to Corollary \ref{cor:BGG},  given $r\in\Delta$, and setting $I=\Delta\setminus\{r\}$, the cohomology of the Picard number one variety $\cD(r)$ 
is the quotient modulo $\cI$ of the subring of $\R[x_1, \dots,x_{n+\epsilon}]$ invariant by the subgroup $W(I)$ of $W$ generated by all the reflections $s_i$, $i\neq r$. Note that the varieties $\DD_n(n-1)$ and $\DD_n(n)$ are isomorphic; hence we may assume that $\cD(r)\neq \DD_n(n-1)$. Then $W(I)$ is the product of the two subgroups $LW(I),RW(I)\subset W(I)$, generated by $\{s_i,\,\,i<r\}$ and $\{s_i,\,\,i>r\}$, respectively, and we have a decomposition:
$$
\R[x_1, \dots,x_{n+\epsilon}]^{W(I)}\cong\R[x_1, \dots,x_{r}]^{LW(I)}\otimes\R[x_{r+1}, \dots,x_{n+\epsilon}]^{RW(I)}.
$$  
Note also that $LW(I)\cong\Sigma_{r}$, whilst the group $RW(I)$ is equal to $\Sigma_{n-r+1}$ (case $\DA_n$), $\Sigma_{n-r}\ltimes \Z_2^{n-r}$ (cases $\DB_n$, $\DC_n$), or $\Sigma_{n-r}\ltimes \Z_2^{n-r-1}$ (case $\DD_n$).

Let us denote by $e_i$ the symmetric elementary polynomial of degree $i$. A set of generators of the  polynomials invariant by $W(I)$ is given in Table \ref{tab:invariant}:

\begin{table}[h!]\centering
\begin{tabular}{|c|c|c|}\hline
Variety  & Invariant polynomials & $i$\\ \hline
& &\\[-12pt] 
\multirow{2}{*}{$\DA_n(r)$}&$q_i:=e_i(x_1, \dots, x_r)$& $1, \dots, r$  \\[2pt]
&$s_{i}= e_i(x_{r+1}, \dots, x_{n+1})$&$1, \dots, n-r+1$\\[3pt]\hline
& &\\[-12pt]
\multirow{2}{*}{$\DB_n(r),\DC_n(r)$}&$q_i:=e_i(x_1, \dots, x_r)$& $1, \dots, r$ \\[2pt]
&$k_{2i}= (-1)^ie_i(x_{r+1}^2, \dots, x_n^2)$&$1, \dots, n-r$\\[3pt]\hline
& &\\[-12pt]
\multirow{3}{*}{$\DD_n(r)$ ($r\neq n-1$)}&$q_i:=e_i(x_1, \dots, x_r)$&$1, \dots, r$\\[2pt]
&$k_{2i}= (-1)^ie_i(x_{r+1}^2, \dots, x_n^2)$&$1, \dots, n-r-1$\\[2pt]
&$\eta_{n-r} = x_{r+1}\cdots x_n$&\\[3pt]
\hline
\end{tabular}\caption{$W(I)$-invariant polynomials, for $I=\Delta\setminus\{r\}$}
\label{tab:invariant}
\end{table}
%
%
%
%
%
%
%

Set $q_0,s_0,k_0:=1$, and, in the case $\DD_n(r)$, $k_{2(n-r)}:=(-1)^{n-r}\eta_{n-r}^2$, and define the polynomials $$q(t) = \sum_{i=0}^rq_{i}t^{i},\qquad s(t) = \sum_{i=0}^{n-r}s_{i}t^{i}, \qquad k(t)= \sum_{i=0}^{n-r}k_{2i}t^{2i}.$$ 
We have, in the case $\DA_n(r)$:
$$Q(t)=\prod_{i=1}^r(1+tx_i)\prod_{i=r+1}^{n+1}(1+tx_i)= q(t)s(t)$$
In the cases $\DB_n(r)$, $\DC_n(r)$ and $\DD_n(r)$:  
$$K(t) =\prod_{i=1}^r(1+tx_i)\prod_{i=1}^r(1-tx_i)\prod_{i=r+1}^{n}(1-t^2x_i^2)=q(t)q(-t)k(t)$$

So we get the presentation for the cohomology rings described in Table \ref{tab:picardonerings}.
\begin{table}[h]\centering
\begin{tabular}{|c|c|c|}\hline
Variety &  Cohomology ring \\ \hline 
& \\[-9pt]
$\DA_n(r)$& $\HH^\bullet(X)=  \dfrac{ \R[q_1,\dots,q_r,s_{1},\dots,s_{n-r+1}]}{\big(\Coeff_+(q(t)s(t))\big)}$\\[12pt] 
$\DB_n(r)$, $\DC_n(r)$ & $\HH^\bullet(X)=  \dfrac{ \R[q_1,\dots,q_r,k_{2},\dots,k_{2(n-r)}]}{\big(\Coeff_+ (q(t)q(-t) k(t))\big)}$ \\[12pt] 
$\DD_n(r)$ ($r\not=n-1$)&  $\HH^\bullet(X)=  \dfrac{ \R[q_1,\dots,q_r,k_{2},\dots,k_{2(n-r-1)}, \eta_{n-r}]}{\big(\Coeff_+ (q(t)q(-t) k(t))\cup\{q_r\eta_{n-r}\}\big)}$\\[12pt] 
\hline
\end{tabular}\caption{Cohomology rings}
\label{tab:picardonerings}
\end{table}

\subsection{Geometrical interpretation}\label{ssec:geominter}

Let $V$ be the natural representation of the Lie algebras of type $\cD=\DA_n,\DB_n,\DC_n,\DD_n$ (which has dimension $N=n+1,2n+1,2n$, and $2n$, respectively), and let $\P(V)$ be its Grothendieck projectivization. We fix an index $r\in\Delta$, and assume, as usual $\cD(r)\neq\DD_n(n-1)$. Then, every variety $\cD(r)$ can be embedded in the Grassmannian of $(r-1)$-dimensional projective subspaces in $\P(V)$, $\G(r-1,\P(V))$. Note that the embedding is given by the ample generator of $\Pic(\cD(r))$, except in the cases $\DB_n(n)$ and $\DD_n(n)$ (in which the embedding is given by the second tensor power of it). In any case, the restriction to $\cD(r)$ of the corresponding universal quotient bundle, $\cQ$, provides a surjective map:
$$
\cO_{\cD(r)}\otimes V\lra\cQ,
$$
whose kernel we denote by $\cS^\vee$.

In the cases $\DB_n$, $\DD_n$ (respectively, $\DC_n$), $V$ supports a nondegenerate quadratic (respectively, skew-symmetric) form $\omega:V^\vee\to V$, with respect to which the vector subspaces parametrized by $\cD(r)$ are isotropic. In other words, the composition:
$$
\cQ^\vee\lra\cO_{\cD(r)}\otimes V^\vee\lra \cO_{\cD(r)}\otimes V\lra \cQ
$$
is zero, and we have a surjection $\cS\lra \cQ$, whose kernel we denote by $\cK$. Summing up, we have a commutative diagram, with short exact rows and columns:

\begin{equation}
\xymatrix@R=10mm{\cQ^\vee\ar@{=}[r]\ar[d]&\cQ^\vee\ar[d]\ar[r]&0\ar[d]\\\cS^\vee\ar[r]\ar[d]&\cO\otimes V^\vee\stackrel{\omega}{\cong} \cO\otimes V\ar[d]\ar[r]&\cQ\ar@{=}[d]\\\cK^\vee\cong\cK\ar[r]&\cS\ar[r]&\cQ}
\label{eq:diagram1}
\end{equation} 

Note that, setting $I=\Delta \setminus \{r\}$, $I_+:=\{i\in\Delta|\,\,i> r\}$, $q_i=e_i(x_1,\dots,x_r)$  as in Table  \ref{tab:invariant}, and recalling that: 
$$
\HH^\bullet(G/P(I))=\dfrac{\R[q_1, \dots,q_{r}]\otimes \R[x_{r+1}, \dots,x_{n+\epsilon}]^{RW(I)}}{\cI},
$$
Corollary \ref{cor:BGG} allows us to write:
$$
\begin{array}{rl}\vspace{7pt}
\HH^\bullet(G/P(I_+))\hspace{-0.2cm}&\cong\dfrac{\R[x_1, \dots,x_{r}]\otimes\R[x_{r+1}, \dots,x_{n+\epsilon}]^{RW(I)}
}{\cI}\\\vspace{7pt}
&\cong\dfrac{\R[q_1, \dots,q_{r}]\otimes\R[x_{r+1}, \dots,x_{n+\epsilon}]^{RW(I)}\otimes \R[x_1, \dots,x_{r}]
}{\cI+\big(\{e_i(x_1,\dots,x_r)-q_i\}_i\big)}\\\vspace{7pt}
&\cong \dfrac{\HH^\bullet(G/P(I))\otimes\R[x_1, \dots,x_{r}]}{\big(\{e_i(x_1,\dots,x_r)-q_i\}_i\big)}.
\end{array}
$$
The bundle $\rho_{I_+,I}:G/P(I_+)\lra G/P(I)$ is the complete flag bundle (of type $\DA_{r-1}$) associated with the Grothendieck projectivization of the universal bundle $\cQ$ on $G/P(I)$. The following statement is based on the well known splitting principle for vector bundles (see for instance \cite[Section~5]{Fult92}): 

\begin{proposition}\label{prop:qsarechern}
The elements $q_i\in\HH^\bullet(G/P(I))$ are the  Chern classes of the universal quotient bundle $\cQ$  on $G/P(I)$. 
\end{proposition}

\begin{proof}
%
%
Following \cite[Section~5]{Fult92}, the pullback bundle $\rho_{I_+,I}^*(\cQ)$ fits in a sequence of vector bundles and surjective maps
$$
\cQ_r:=\rho_{I_+,I}^*(\cQ)\lra\cQ_{r-1}\lra\dots\lra\cQ_{1}\lra\cQ_{0}:=0,
$$
where, for every $i<r$, the bundle $\cQ_i$ is defined as the pullback to $G/P(I_+)$ of:
\begin{itemize}
\item the universal quotient bundle on $\DD_n(n)\hookrightarrow\G(n-2,\P(V))$, if $(\cD(r),i)=(\DD_n(n),n-1)$, or
\item the universal quotient bundle on $\cD(i)\hookrightarrow\G(i-1,\P(V))$, otherwise.
\end{itemize}
From the sequence, we have an equality of Chern polynomials:
$$
c_t(\rho_{I_+,I}^*(\cQ))=\prod_ic_t(\ker(\cQ_i\to\cQ_{i-1}))=\prod_i\big(1+(c_1(\cQ_i)-c_1(\cQ_{i-1}))t\big).
$$
We consider now the fibers $\Gamma_j$ of the elementary contractions of the flag manifold $G/B$ (see Section \ref{ssec:partial}). By construction, the classes $L_i:=c_1(\cQ_i)$ satisfy 
$$L_i \cdot \Gamma_j=\left\{\begin{array}{l}2, \,\,\,\,\,\quad\mbox{if }(\cD,i,j)= (\DB_n,n,n),(\DD_n,n,n),\\
1, \,\,\,\,\,\quad\mbox{if }(\cD,i,j)= (\DD_n,n-1,n),\\
\delta_{i,j},\,\,\,\,\,\mbox{otherwise}.
\end{array}\right.$$ 
In the case $\DA_n$, we also set $L_{n+1}:=0$. From this data one can write the $\alpha_i$'s in terms of the $x_j$'s, by means of Table \ref{tab:bases}, and in terms of the $L_j$'s using the fact that $\alpha_i\cdot\Gamma_j=-K_i\cdot\Gamma_j$ is the coefficient in the position $(i,j)$ of the Cartan matrix of $\cD$, to conclude that $L_i-L_{i-1}=x_i$ for every $i$. 
For instance, in the case $\DD_n(n)$, this follows from the equality:
$$
\left(\begin{array}{ccccc}1&-1&\ldots&0&0\\0&\ddots&\ddots&&0\\\vdots&&\ddots&\ddots&\vdots\\0&\ldots&\ldots&1&-1\\0&\ldots&\ldots&1&1\end{array}\right)\hspace{-0.15cm}\left(\hspace{-0.2cm}\begin{array}{c}x_1\\\vdots\\\vdots\\x_{n-1}\\x_n\end{array}\hspace{-0.2cm}\right)\hspace{-0.05cm}=\hspace{-0.05cm}
\left(\begin{array}{ccccc}2&-1&\ldots&0&0\\-1&\ddots&\ddots&&0\\\vdots&\ddots&2&-1&-1\\0&\ldots&-1&2&0\\0&\ldots&-1&0&2\end{array}\right)\hspace{-0.15cm}
\left(\hspace{-0.2cm}\begin{array}{c}L_1\\\vdots\\\vdots\\L_{n-1}-\frac{1}{2}L_n\\\frac{1}{2}L_n\end{array}\hspace{-0.2cm}\right)
$$

Now one may finally assert that the Chern classes of $\rho_{I_+,I}^*(\cQ)$ are the elementary symmetric polynomials in the elements $x_i$, from which the statement follows.
\end{proof}

\begin{corollary}\label{cor:qsarechern}
With the same notation as above:
$$
q(t)=c_t(\cQ),\quad s(t)=c_t(\cS^\vee),\quad\mbox{and, \,\,for $\cD\neq\DA_n$, \,\,} k(t)=c_t(\cK).
$$
\end{corollary}

\begin{proof}
The first equality follows from Proposition \ref{prop:qsarechern}. Then, 
 in the case $\DA_n$ we also have $c_t(\cS^\vee)=s(t)$ by Table \ref{tab:picardonerings}. 

In the cases $\DB_n$, $\DC_n$ and $\DD_n$, 
by the diagram (\ref{eq:diagram1}) above we have the equalities of Chern polynomials: $ c_t(\cS) c_t(\cQ^\vee)=1$ and $c_t(\cK)c_t(\cQ)=c_t(\cS)$.
Multiplying the second equality by $c_t(\cQ^\vee)$ we get $c_t(\cK)c_t(\cQ)c_t(\cQ^\vee)=1$. 

On the other hand  (see Table \ref{tab:picardonerings})  in  $\HH^\bullet(X)[t]$ we have the equality of polynomials $k(t)c_t(\cQ)c_t(\cQ^\vee)=k(t)q(t)q(-t)=1$. Then, setting $p(t)=\sum_{i}p_it^i:=c_t(\cK)-k(t)\in \HH^\bullet(X)[t]$, we have
$$
p(t)c_t(\cQ)c_t(\cQ^\vee)=1\in \HH^\bullet(X)[t].
$$
By construction $p_0=0$. Hence, since the constant term of $c_t(\cQ)c_t(\cQ^\vee)$ is equal to $1$, we can easily show, by induction on $i$, that $p_i=0$, for all $i$, and conclude that $c_t(\cK)=k(t)$.
\end{proof}

%


\section{Morphisms to rational homogeneous manifolds of Picard number one}\label{sec:morRH}

In this section we will use the above description on the cohomology of rational homogeneous spaces to obtain conditions for the existence of non constant morphisms from projective manifolds to rational homogeneous spaces. Let us start by recalling the definition of good divisibility from \cite{Pan} and a refined version of it --for effective cycles-- that we will use later on.

\begin{remark}
By definition, an element $x_i\in \HH^{2i}(X)$ is called  {\em effective} (and we write $x_i\geq 0$) if it can be expressed as a non negative real linear combination of classes of subvarieties of codimension $i$ in $X$. 
\end{remark}

\begin{definition} 
We say that a complex projective manifold $M$ has {\it (effective) good divisibility up to degree} $s$, and we write ($\ed(M)=s$) $\gd(M)=s$, if $s$ is the maximum positive integer such that, given (effective clasess) $x_i \in \HH^{2i}(M)$ and  $x_j \in \HH^{2j}(M)$ with $i+j \le s$ and $x_i  x_j=0$, then we have $x_i = 0$ or $x_j = 0$. 
\end{definition}

It is clear from the definition that $\gd(M)\leq \ed(M)\leq \dim(M)$.

\begin{example}
If $\HH^{2i}(M)\cong\R$ for every $i$ then clearly $\ed(M)=\gd(M)=\dim(M)$ (in fact they are equivalent by the Hard Lefschetz theorem). This is the case, for instance, of the projective spaces $\DA_n(1)$, and the odd dimensional quadrics $\DB_n(1)$.  
If $\HH^{2i}(M)\cong\R$ for every $i\leq s<\dim(M)$, then we may conclude that $\ed(M)\geq\gd(M)\geq s+1$. For instance, in the case of the even
dimensional quadric $\Q^{2n}=\DD_{n+1}(1)$, we have:
$$
\HH^\bullet(\DD_{n+1}(1))=\dfrac{\R[q_1,\eta_{n}]}{(\eta_{n}^2-(-1)^nq_1^{2n},q_1\eta_{n})},
$$
so that $\HH^{2i}(\Q^{2n})\cong\R$ for $i<n$ and, in particular, $\ed(\Q^{2n})\geq \gd(\Q^{2n})\geq n$. On the other hand  
we have a relation $q_1\eta_n=0$ in $\HH^{2(n+1)}(\Q^{2n})$, which tells us that $\gd(\Q^{2n})=n$ (note that this implies that \cite[Lemma 5.4]{Pan} is incorrect). Since $q_1$ is an ample generator of $\HH^2(\Q^{2n})$,  the condition $q_1\eta_n=0$ implies that $\pm\eta_n$  cannot be effective, and we have $\ed(\Q^{2n})>n$. Moreover, since $\HH^{2i}(\Q^{2n})=\R q_1^i$ for $i<n$, it follows that we have only one nontrivial relation in degree $n+i$, namely $q_1^i\eta_n=0$, and so $\ed(\Q^{2n})\geq 2n-1$. On the other hand, in codimension $n$, the two classes $\alpha,\beta$ of $(n-1)$-dimensional linear spaces in $\Q^{2n}$ satisfy, either $\alpha \beta=0$ (if $n$ is even), or $\alpha^2=\beta^2=0$ (if $n$ is odd), and we may finally conclude that $\ed(\Q^{2n})=2n-1$ (moreover one may show, from the above equations, that $q_1^n=\alpha+\beta$, and $\eta_n=\alpha-\beta$). \qed
\end{example}

Another example in which $\gd(M)\neq \ed(M)$ is the Grassmannian of lines in $\P^m$, $\DA_m(2)=\G(1,m)$, $m\geq 3$. In this case, a straightforward computation provides a presentation:
$$
\HH^\bullet(\DA_m(2))=\dfrac{\R[q_1,q_2]}{(q_1s_{m-1}+q_2s_{m-2},q_2s_{m-1})},
$$
where $s_{m-1}, s_{m-2}$ can be written as polynomials in $q_1,q_2$ as follows:
\begin{equation}\label{eq:matrixgrass}
\left(\begin{array}{c}s_{m-1}\\s_{m-2}\end{array}\right)=\left(\begin{array}{cc}-q_1&-q_2\\1&0\end{array}\right)^{m-1}\left(\begin{array}{c}1\\0\end{array}\right).
\end{equation}
For instance, in the case $m=3$, the first generator of the ideal above provides a relation $q_1(q_1^2-2q_2)=0$ in the cohomology ring, which provides $\gd(\DA_3(2))=2$; this was already known from the fact that  $\DA_3(2)$ is a $4$-dimensional quadric. More generally, we have: 

\begin{lemma}
The Grassmannian $\DA_m(2)=\G(1,m)$, $m\geq 3$ has good divisibility equal to $m-1$.
\end{lemma}

\begin{proof}
By means of (\ref{eq:matrixgrass}), the relation $q_1s_{m-1}+q_2s_{m-2}=0$ can be written in terms of $q_1,q_2$ as:
$$
\left(\begin{array}{cc}1&0\end{array}\right)\left(\begin{array}{cc}q_1&q_2\\-1&0\end{array}\right)^{m}\left(\begin{array}{c}1\\0\end{array}\right)=0.
$$
One can prove, recursively, that this is either a homogeneous polynomial in $q_2$ and $q_1^2$ for $m$ even, or a multiple of $q_1$ when $m$ is odd. In any case, it is reducible, and  we have exactly $\gd(\DA_m(2))=m-1$. 
\end{proof}

In order to calculate the effective good divisibility of $\DA_m(2)$, we will use standard Schubert calculus, for which we introduce the following notation.

\begin{notation} {\rm We denote by $\sigma_{a_1,a_2}$ the Schubert cycle in $\G(1,m)$
determined by the sequence $(a_1,a_2)$, where $m-1 \ge a_1 \ge a_2 \ge 0$; this cycle  has codimension $|a|:=a_1+a_2$.
Taking a partial flag $\P^{m-1-a_1}\subset \P^{m-a_2}$, $\sigma_{a_1,a_2}$ is the cohomology class of the set of lines meeting $\P^{m-1-a_1}$, and contained in  $\P^{m-a_2}$. }
\end{notation}

\begin{lemma}\label{lem:glines2} The Grassmannian $\DA_m(2)=\G(1,m)$, $m\geq 3$ has effective good divisibility equal to $m$.
\end{lemma}

\begin{proof}
Note that $q_2\geq 0$ and either $s_{m-1} 	\geq 0$ or $-s_{m-1}\geq 0$, hence our presentation of the cohomology ring of $\DA_m(2)$ tells us already that $\ed(\DA_m(2))\leq m$. For the converse, let $\Gamma_k \in \HH^{2k}(\G(1,m))$, $\Delta_{l} \in  \HH^{2l}(\G(1,m))$ be two effective nonzero classes, such that $\Gamma_k  \Delta_l =0$; we want to show that $k+l \ge m+1$.
Since the cones of effective classes in $\G(1,m)$ are polyhedral cones generated by Schubert classes (see \cite[Corollary of Theorem 1]{FMSS} for a more general statement), we may write:
$$
\Gamma_k=\sum_{i=0}^{\lfloor k/2\rfloor}a_i\sigma_{k-i,i},\quad \Delta_{l}=\sum_{j=0}^{\lfloor l/2\rfloor}b_j\sigma_{l-j,j},
$$
where the $a_i$'s and the $b_j$'s are non negative integers.

Every intersection $\sigma_{k-i,i} \sigma_{l-j,j}$ is a combination  of Schubert cycles, with nonnegative coefficients (due to the Littlewood--Richardson rule, see \cite[Lemma 14.5.3 and Section 14.7]{Fult}), hence it is then enough to prove the statement for $\Gamma_k = \sigma_{a_1,a_2}$ and $\Delta_l=\sigma_{b_1,b_2}$, where we assume, up to order exchange, that $a_1-a_2 \ge b_1-b_2$.
By \cite[Proposition 4.11]{3264} we have
$$ \sigma_{a_1,a_2}   \sigma_{b_1,b_2} = \sum_{\substack{|c|=k+l \\ a_1+b_1\ge c_1 \ge a_1+b_2}} \sigma_{c_1,c_2};$$
since $(k+l)/2 \le c_1 \le m-1$ the above sum has no summands if and only if  $a_1+b_2 \ge m$, which, recalling that $a_1 \le m-1$, provides $b_2\ge 1$. Therefore $b_1 \ge b_2>0$, and so $k+l = |a|+|b| \ge m+1$.
\end{proof}

\begin{definition}\label{def:cox}
Let $\pi:Y \to X$ be a $G/B$-bundle. The {\em Coxeter number of} $\pi$, denoted by $\Cox(\pi)$, is defined as the Coxeter number of the Dynkin diagram of $G$. For a rational homogeneous manifold of Picard number one, different from $\P^1$, we define:
$$\Cox(X):=\min\{\Cox(\pi)|\,\,\pi \mbox{ universal flag bundle on }X\},
$$
and set $\Cox(\P^1):=+\infty$.
\end{definition}

\begin{proposition}\label{prop:classical}
Let $M$ be a complex projective manifold and $\cD(r)$ a rational homogeneous manifold of Picard number one of classical type, satisfying that $\ed(M) \ge \Cox(\cD)$. In the case $\cD=\DD_n$, assume also that $2\gd(M) > \Cox(\cD)$. Then there are no nonconstant morphisms from $M$ to $\cD(r)$.
\end{proposition}

\begin{proof}
Since $\DD_{n}(n) \simeq \DD_n(n-1)$, we can assume $r \neq n-1$ in case $\DD_n(r)$. Let $\phi:M \to \cD(r)$ be a morphism, and set 
$$\lambda(t) = \phi^*q(t), \quad \mu(t)=\begin{cases}\phi^*s(t)& \textrm{if~} \cD=\DA_n\\\phi^*(q(-t)k(t)) & \textrm{if~} \cD=\DB_n, \DC_n, \DD_n \\ \end{cases}$$

First we claim that
 $\deg(\lambda(t))+ \deg(\mu(t)) \le h(\cD)$, which follows from the definition in cases $\DA_n,\DB_n$ and $\DC_n$. For case $\DD_n$ we have $\phi^*q_r\phi^*\eta_{n-r}=0$  (see Table (\ref{tab:picardonerings})) which, together with the hypothesis on $\gd(M)$, tells us that either $\phi^*q_r=0$ or $\phi^*\eta_{n-r}=0$. Since the leading term of $\lambda(t)$ (resp. $\mu(t)$) is $\phi^*(q_r)t^r$ (resp. $(-1)^{r-1}\phi^*(q_{r-1})\phi^*(\eta_{n-r})^2t^{2n-r-1}$) then, either $\deg(\lambda)\leq r-1$ or $\deg(\mu)\leq 2n-r-2$. In any case the sum of their degrees is at most $2n-2=\Cox(\DD_n)$.
 
Write $\lambda(t)=\sum_i\lambda_it^i$, and $\mu(t)=\sum_i\mu_it^i$, and let $  \imath$ and $  \jmath$ be the maximum indexes for which $\lambda_{  \imath} \not = 0$ and  $\mu_{  \jmath} \not = 0$. From Table (\ref{tab:picardonerings}), we have $\lambda(t)\mu(t)=1$, and, as we have already noticed that $  \imath +   \jmath \le \Cox(\cD)$. If $  \imath   +\jmath \not =0$, we have
$$\lambda_{  \imath}\mu_{  \jmath}\,=0;$$
since, from Corollary \ref{cor:qsarechern},
$\lambda_{  \imath}=c_{  \imath}(\phi^*\cQ)$ is an effective class and $\mu_{  \jmath}=c_{  \jmath}(\phi^*\cS^\vee)$ is either effective or antieffective, we reach a contradiction with the assumption $\ed(M) \ge h(\cD)$. We conclude that $\lambda$ and $\mu$ are constant, so that, in particular, the pullback by $\phi$ of the ample line bundle $\det(\cQ)$ is numerically trivial, and $\phi$ is constant.  
\end{proof}


\section{Splitting Conjecture $\bm{(\Cox)}$ for groups of classical type}\label{sec:main}

\begin{theorem}\label{thm:main}
Let $X$ be a rational homogenous manifold of Picard number one, quotient of a simple algebraic  group $\ol G$ and $\cM$ its family of $\ol G$-isotropic lines. Let $G$ be another simple algebraic group, and let $\pi: Y \to X$ be a $G/B$-bundle, uniform with respect to $\cM$.  Assume that both $\ol G$ and $G$ are of classical type and that $\Cox(X)>\Cox(\pi)$. Then $\pi: Y \to X$ is diagonalizable.
\end{theorem}

\begin{proof} 
Let us write $X$ as $X=\ol{\cD}({\jmath})$, where $\ol{\cD}$ is the Dynkin diagram of $\ol G$ and, following Section \ref{ssec:redunif}, denote by $I_0$ the set of nodes of the Dynkin diagram $\cD$ of $G$ for which the tag of the uniform bundle $\pi$ is equal to zero. 
We will show that, for every $r\in \Delta$, there are no nonconstant maps from $\cM_x$ to $\cD(r)$, and conclude by Proposition \ref{prop:diagcrit}; for this purpose, we will apply Proposition \ref{prop:classical}.

Let us assume first that ${\jmath}$ is an extremal node, that is, that the subdiagram $\ol{\cD} \setminus\{\ol\jmath\}$ is connected. The possibilities are listed in Table \ref{tab:gded} below.

\renewcommand*{\arraystretch}{1.2}
\begin{table}[h!]\centering
\begin{tabular}{|c|c|c H|c|c|}\hline
$X$   &$\Cox(X)$&  $ \cM_x$&$\dim \cM_x $  & $\gd(\cM_x)$ & $\ed(\cM_x)$ \\ 
\hline
$\DA_n(1)$ &$n$&$\DA_{n-1}(1)$ &$n-1$ & $n-1$ & $n-1$ \\
\hline
$\DB_n(1)$ &$2n-2$&$\DB_{n-1}(1)$ &$2n-3$ & $2n-3$ & $2n-3$ \\
\hline
$\DB_n(n)$ &$n$&$\DA_{n-1}(n-1)$ &$n-1$ & $n-1$ & $n-1$\\
\hline
$\DC_n(1)$ &$2n-2$&$\DC_{n-1}(1)$ &$2n-3$ & $2n-3$ & $2n-3$ \\
\hline
$\DC_n(n)$ &$n$&$ \DA_{n-1}(n-1)$&$n-1$ & $n-1$ & $n-1$\\
\hline
$\DD_n(n)$ &$n$&$ \DA_{n-1}(n-2)$& $2n-4$ & $n-2$ & $n-1$\\
\hline
$\DD_n(1)$ &$2n-4$& $\DD_{n-1}(1)$ & $2n-4$ & $n-2$ & $2n-5$\\
\hline
\end{tabular} 
\caption{Picard number one RHS corresponding to an extremal node}
\label{tab:gded}
\end{table}

In all the cases we have 
$\ed(\cM_x) \geq \Cox(X)-1 \ge \Cox(\pi)=\Cox(\cD),$ and
$2 \gd (\cM_x)  \ge \Cox(X) >\Cox(\pi)=\Cox(\cD)$, so we can conclude by Proposition \ref{prop:classical}. 

Assume now that the node ${\jmath}$ is not extremal and set $\{{\jmath}_s\}_s=\ngb(\jmath)$. In this case, $\cM_x$ is isomorphic to a product of rational homogeneous spaces, one for each neighboring node of $\jmath$, that we will describe as follows. 

For every neighboring node ${\jmath}_t$, we consider the following rational homogeneous varieties:
$$
\cU_t:=\ol{\cD}(\{{\jmath}\}\cup\{{\jmath}_s\}_{s\neq t}),\qquad \cU'_t:=\ol{\cD}(\{{\jmath}_s\}_{s\neq t}),
$$
and the corresponding contractions:
$$\xymatrix@R=3mm{&&X\\\cU\ar@/^3mm/[urr]^{q}\ar[r]\ar[rd]&\cU_t\ar[ru]\ar[rd]^{p_t}&\\&\cM\ar[r]&\cU'_t}
$$
Denoting by $J_t$ the union of $\{\jmath\}$ and the connected component of $\dis({\jmath})$ containing ${\jmath}_t$, the fibers $F_t$ of the map $p_t:\cU_t \to \cU'_t$ are rational homogeneous manifolds of Picard number one, isomorphic to $\ol{\cD}_{J_t}(\jmath)$. Every $F_t$ can be thought of as a subvariety of $X$, covered by a subfamily $q_{t}:\cU_{F_t}\to \cM_{F_t}$ 
of $\ol{G}$-isotropic lines, 
that we may identify with
$$
\ol{\cD}_{J_t}(\{\jmath,\jmath_t\})\lra\ol{\cD}_{J_t}(\jmath_t)
$$
Now, for every $x\in F_t$, we will have $\cM_{F_t,x}\cong \ol{\cD}_{J_t\setminus\{\jmath\}}(\jmath_t)$, which is a fiber of the natural map $\cU\to \cU_t$. Finally, we have:
$$
\cM_x\cong \prod_{t\in\ngb(\jmath)}\cM_{F_t,x}.
$$

Note that, by construction, every  $F_t$ is defined by an extremal node, $\Cox(F_t)\geq \Cox(X)>\Cox(\pi)$, and the restriction of the bundle $Y$ to $F_t$ is uniform with respect to the family $\cM_{F_t}$. Hence we may apply our previous argument to $F_t$, and claim that the section $s_I$ is constant on the $\cM_{F_t,x}$, $x\in F_t$, hence on the fibers of $\cU \to \cU_t$. Since this holds for every $t$, we conclude that $s_I$ is constant on the fibers of $q$. 

To help the reader following the above proof we provide the corresponding marked Dynkin diagrams for  $\ol{\cD}=\DB_5$, $\jmath=2$, $\jmath_t=3$.

$$
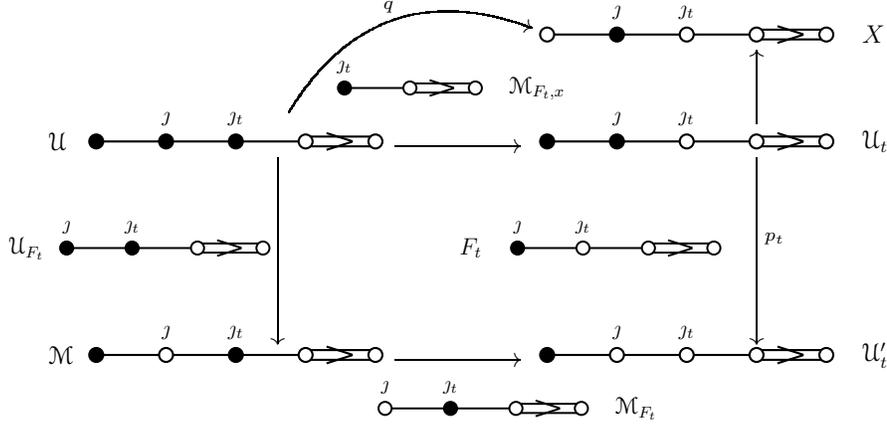
\xymatrix@R=5mm@C=1.2cm{&&&&&\ifx\du\undefined
  \newlength{\du}
\fi
\setlength{\du}{3.3\unitlength}
\begin{tikzpicture}[transform canvas={scale=0.8}]
\pgftransformxscale{1.000000}
\pgftransformyscale{1.000000}

\definecolor{dialinecolor}{rgb}{0.000000, 0.000000, 0.000000} 
\pgfsetstrokecolor{dialinecolor}
\definecolor{dialinecolor}{rgb}{0.000000, 0.000000, 0.000000} 
\pgfsetfillcolor{dialinecolor}


\pgfsetlinewidth{0.300000\du}
\pgfsetdash{}{0pt}
\pgfsetdash{}{0pt}

\pgfpathellipse{\pgfpoint{-26\du}{0\du}}{\pgfpoint{1\du}{0\du}}{\pgfpoint{0\du}{1\du}}
\pgfusepath{stroke}
\node at (-26\du,0\du){};

\pgfpathellipse{\pgfpoint{-16\du}{0\du}}{\pgfpoint{1\du}{0\du}}{\pgfpoint{0\du}{1\du}}
\pgfusepath{stroke}
\node at (-16\du,0\du){};
\pgfpathellipse{\pgfpoint{-16\du}{0\du}}{\pgfpoint{1\du}{0\du}}{\pgfpoint{0\du}{1\du}}
\pgfusepath{fill}
\node at (-16\du,0\du){};

\pgfpathellipse{\pgfpoint{-6\du}{0\du}}{\pgfpoint{1\du}{0\du}}{\pgfpoint{0\du}{1\du}}
\pgfusepath{stroke}
\node at (-6\du,0\du){};

\pgfpathellipse{\pgfpoint{4\du}{0\du}}{\pgfpoint{1\du}{0\du}}{\pgfpoint{0\du}{1\du}}
\pgfusepath{stroke}
\node at (4\du,0\du){};

\pgfpathellipse{\pgfpoint{14\du}{0\du}}{\pgfpoint{1\du}{0\du}}{\pgfpoint{0\du}{1\du}}
\pgfusepath{stroke}
\node at (14\du,0\du){};

\pgfsetlinewidth{0.300000\du}
\pgfsetdash{}{0pt}
\pgfsetdash{}{0pt}
\pgfsetbuttcap

{\draw (-25\du,0\du)--(-17\du,0\du);}
{\draw (-15\du,0\du)--(-7\du,0\du);}
{\draw (-5\du,0\du)--(3\du,0\du);}
{\draw (4.65\du,0.7\du)--(13.35\du,0.7\du);}
{\draw (4.65\du,-0.7\du)--(13.35\du,-0.7\du);}


{\pgfsetcornersarced{\pgfpoint{0.300000\du}{0.300000\du}}
\pgfsetstrokecolor{dialinecolor}
\draw (7\du,-1.2\du)--(10.8\du,0\du)--(7\du,1.2\du);}

\node[anchor=west] at (18\du,0\du){\Large $X$};


\node[anchor=south] at (-16\du,1.1\du){$\large \jmath$};

\node[anchor=south] at (-6\du,1.1\du){$\large \jmath_t$};



\end{tikzpicture} \\
&&{\ifx\du\undefined
  \newlength{\du}
\fi
\setlength{\du}{3.3\unitlength}
\begin{tikzpicture}[transform canvas={scale=0.75}]
\pgftransformxscale{1.000000}
\pgftransformyscale{1.000000}

\definecolor{dialinecolor}{rgb}{0.000000, 0.000000, 0.000000} 
\pgfsetstrokecolor{dialinecolor}
\definecolor{dialinecolor}{rgb}{0.000000, 0.000000, 0.000000} 
\pgfsetfillcolor{dialinecolor}


\pgfsetlinewidth{0.300000\du}
\pgfsetdash{}{0pt}
\pgfsetdash{}{0pt}

\pgfpathellipse{\pgfpoint{-6\du}{0\du}}{\pgfpoint{1\du}{0\du}}{\pgfpoint{0\du}{1\du}}
\pgfusepath{stroke}
\node at (-6\du,0\du){};
\pgfpathellipse{\pgfpoint{-6\du}{0\du}}{\pgfpoint{1\du}{0\du}}{\pgfpoint{0\du}{1\du}}
\pgfusepath{fill}
\node at (-6\du,0\du){};

\pgfpathellipse{\pgfpoint{4\du}{0\du}}{\pgfpoint{1\du}{0\du}}{\pgfpoint{0\du}{1\du}}
\pgfusepath{stroke}
\node at (4\du,0\du){};

\pgfpathellipse{\pgfpoint{14\du}{0\du}}{\pgfpoint{1\du}{0\du}}{\pgfpoint{0\du}{1\du}}
\pgfusepath{stroke}
\node at (14\du,0\du){};

\pgfsetlinewidth{0.300000\du}
\pgfsetdash{}{0pt}
\pgfsetdash{}{0pt}
\pgfsetbuttcap

{\draw (-5\du,0\du)--(3\du,0\du);}
{\draw (4.65\du,0.7\du)--(13.35\du,0.7\du);}
{\draw (4.65\du,-0.7\du)--(13.35\du,-0.7\du);}


{\pgfsetcornersarced{\pgfpoint{0.300000\du}{0.300000\du}}
\pgfsetstrokecolor{dialinecolor}
\draw (7\du,-1.2\du)--(10.8\du,0\du)--(7\du,1.2\du);}

\node[anchor=west] at (18\du,0\du){\Large$\cM_{F_t,x}$
};

\node[anchor=south] at (-6\du,1.1\du){$\large \jmath_t$};



\end{tikzpicture} }&&&\\
& \ifx\du\undefined
  \newlength{\du}
\fi
\setlength{\du}{3.3\unitlength}
\begin{tikzpicture}[transform canvas={scale=0.8}]
\pgftransformxscale{1.000000}
\pgftransformyscale{1.000000}

\definecolor{dialinecolor}{rgb}{0.000000, 0.000000, 0.000000} 
\pgfsetstrokecolor{dialinecolor}
\definecolor{dialinecolor}{rgb}{0.000000, 0.000000, 0.000000} 
\pgfsetfillcolor{dialinecolor}


\pgfsetlinewidth{0.300000\du}
\pgfsetdash{}{0pt}
\pgfsetdash{}{0pt}

\pgfpathellipse{\pgfpoint{-26\du}{0\du}}{\pgfpoint{1\du}{0\du}}{\pgfpoint{0\du}{1\du}}
\pgfusepath{stroke}
\node at (-26\du,0\du){};
\pgfpathellipse{\pgfpoint{-26\du}{0\du}}{\pgfpoint{1\du}{0\du}}{\pgfpoint{0\du}{1\du}}
\pgfusepath{fill}
\node at (-26\du,0\du){};

\pgfpathellipse{\pgfpoint{-16\du}{0\du}}{\pgfpoint{1\du}{0\du}}{\pgfpoint{0\du}{1\du}}
\pgfusepath{stroke}
\node at (-16\du,0\du){};
\pgfpathellipse{\pgfpoint{-16\du}{0\du}}{\pgfpoint{1\du}{0\du}}{\pgfpoint{0\du}{1\du}}
\pgfusepath{fill}
\node at (-16\du,0\du){};

\pgfpathellipse{\pgfpoint{-6\du}{0\du}}{\pgfpoint{1\du}{0\du}}{\pgfpoint{0\du}{1\du}}
\pgfusepath{stroke}
\node at (-6\du,0\du){};
\pgfpathellipse{\pgfpoint{-6\du}{0\du}}{\pgfpoint{1\du}{0\du}}{\pgfpoint{0\du}{1\du}}
\pgfusepath{fill}
\node at (-6\du,0\du){};

\pgfpathellipse{\pgfpoint{4\du}{0\du}}{\pgfpoint{1\du}{0\du}}{\pgfpoint{0\du}{1\du}}
\pgfusepath{stroke}
\node at (4\du,0\du){};

\pgfpathellipse{\pgfpoint{14\du}{0\du}}{\pgfpoint{1\du}{0\du}}{\pgfpoint{0\du}{1\du}}
\pgfusepath{stroke}
\node at (14\du,0\du){};

\pgfsetlinewidth{0.300000\du}
\pgfsetdash{}{0pt}
\pgfsetdash{}{0pt}
\pgfsetbuttcap

{\draw (-25\du,0\du)--(-17\du,0\du);}
{\draw (-15\du,0\du)--(-7\du,0\du);}
{\draw (-5\du,0\du)--(3\du,0\du);}
{\draw (4.65\du,0.7\du)--(13.35\du,0.7\du);}
{\draw (4.65\du,-0.7\du)--(13.35\du,-0.7\du);}


{\pgfsetcornersarced{\pgfpoint{0.300000\du}{0.300000\du}}
\pgfsetstrokecolor{dialinecolor}
\draw (7\du,-1.2\du)--(10.8\du,0\du)--(7\du,1.2\du);}

\node[anchor=west] at (-34\du,0\du){\Large $\cU$};


\node[anchor=south] at (-16\du,1.1\du){$\large \jmath$};

\node[anchor=south] at (-6\du,1.1\du){$\large \jmath_t$};



\end{tikzpicture} \ar@/^7mm/!<1ex,2ex>;[uurrr]!<-8ex,0ex>^{q}
\ar!<0ex,-2ex>;[dddd]!<0ex,-2ex>&\ar!<1ex,-1ex>;[rr]!<-9ex,-1ex>&&&\ifx\du\undefined
  \newlength{\du}
\fi
\setlength{\du}{3.3\unitlength}
\begin{tikzpicture}[transform canvas={scale=0.8}]
\pgftransformxscale{1.000000}
\pgftransformyscale{1.000000}

\definecolor{dialinecolor}{rgb}{0.000000, 0.000000, 0.000000} 
\pgfsetstrokecolor{dialinecolor}
\definecolor{dialinecolor}{rgb}{0.000000, 0.000000, 0.000000} 
\pgfsetfillcolor{dialinecolor}


\pgfsetlinewidth{0.300000\du}
\pgfsetdash{}{0pt}
\pgfsetdash{}{0pt}

\pgfpathellipse{\pgfpoint{-26\du}{0\du}}{\pgfpoint{1\du}{0\du}}{\pgfpoint{0\du}{1\du}}
\pgfusepath{stroke}
\node at (-26\du,0\du){};
\pgfpathellipse{\pgfpoint{-26\du}{0\du}}{\pgfpoint{1\du}{0\du}}{\pgfpoint{0\du}{1\du}}
\pgfusepath{fill}
\node at (-26\du,0\du){};

\pgfpathellipse{\pgfpoint{-16\du}{0\du}}{\pgfpoint{1\du}{0\du}}{\pgfpoint{0\du}{1\du}}
\pgfusepath{stroke}
\node at (-16\du,0\du){};
\pgfpathellipse{\pgfpoint{-16\du}{0\du}}{\pgfpoint{1\du}{0\du}}{\pgfpoint{0\du}{1\du}}
\pgfusepath{fill}
\node at (-16\du,0\du){};

\pgfpathellipse{\pgfpoint{-6\du}{0\du}}{\pgfpoint{1\du}{0\du}}{\pgfpoint{0\du}{1\du}}
\pgfusepath{stroke}
\node at (-6\du,0\du){};

\pgfpathellipse{\pgfpoint{4\du}{0\du}}{\pgfpoint{1\du}{0\du}}{\pgfpoint{0\du}{1\du}}
\pgfusepath{stroke}
\node at (4\du,0\du){};

\pgfpathellipse{\pgfpoint{14\du}{0\du}}{\pgfpoint{1\du}{0\du}}{\pgfpoint{0\du}{1\du}}
\pgfusepath{stroke}
\node at (14\du,0\du){};

\pgfsetlinewidth{0.300000\du}
\pgfsetdash{}{0pt}
\pgfsetdash{}{0pt}
\pgfsetbuttcap

{\draw (-25\du,0\du)--(-17\du,0\du);}
{\draw (-15\du,0\du)--(-7\du,0\du);}
{\draw (-5\du,0\du)--(3\du,0\du);}
{\draw (4.65\du,0.7\du)--(13.35\du,0.7\du);}
{\draw (4.65\du,-0.7\du)--(13.35\du,-0.7\du);}


{\pgfsetcornersarced{\pgfpoint{0.300000\du}{0.300000\du}}
\pgfsetstrokecolor{dialinecolor}
\draw (7\du,-1.2\du)--(10.8\du,0\du)--(7\du,1.2\du);}

\node[anchor=west] at (18\du,0\du){\Large$\cU_t$};


\node[anchor=south] at (-16\du,1.1\du){$\large \jmath$};

\node[anchor=south] at (-6\du,1.1\du){$\large \jmath_t$};



\end{tikzpicture} \ar!<2.5ex,1ex>;[uu]!<2.5ex,-2ex>\ar!<2.5ex,-2ex>;[dddd]!<2.5ex,-1ex>^(.40){p_t}\\
&&&&&\\
\ifx\du\undefined
  \newlength{\du}
\fi
\setlength{\du}{3.3\unitlength}
\begin{tikzpicture}[transform canvas={scale=0.75}]
\pgftransformxscale{1.000000}
\pgftransformyscale{1.000000}

\definecolor{dialinecolor}{rgb}{0.000000, 0.000000, 0.000000} 
\pgfsetstrokecolor{dialinecolor}
\definecolor{dialinecolor}{rgb}{0.000000, 0.000000, 0.000000} 
\pgfsetfillcolor{dialinecolor}


\pgfsetlinewidth{0.300000\du}
\pgfsetdash{}{0pt}
\pgfsetdash{}{0pt}

\pgfpathellipse{\pgfpoint{-16\du}{0\du}}{\pgfpoint{1\du}{0\du}}{\pgfpoint{0\du}{1\du}}
\pgfusepath{stroke}
\node at (-16\du,0\du){};
\pgfpathellipse{\pgfpoint{-16\du}{0\du}}{\pgfpoint{1\du}{0\du}}{\pgfpoint{0\du}{1\du}}
\pgfusepath{fill}
\node at (-16\du,0\du){};

\pgfpathellipse{\pgfpoint{-6\du}{0\du}}{\pgfpoint{1\du}{0\du}}{\pgfpoint{0\du}{1\du}}
\pgfusepath{stroke}
\node at (-6\du,0\du){};
\pgfpathellipse{\pgfpoint{-6\du}{0\du}}{\pgfpoint{1\du}{0\du}}{\pgfpoint{0\du}{1\du}}
\pgfusepath{fill}
\node at (-6\du,0\du){};

\pgfpathellipse{\pgfpoint{4\du}{0\du}}{\pgfpoint{1\du}{0\du}}{\pgfpoint{0\du}{1\du}}
\pgfusepath{stroke}
\node at (4\du,0\du){};

\pgfpathellipse{\pgfpoint{14\du}{0\du}}{\pgfpoint{1\du}{0\du}}{\pgfpoint{0\du}{1\du}}
\pgfusepath{stroke}
\node at (14\du,0\du){};

\pgfsetlinewidth{0.300000\du}
\pgfsetdash{}{0pt}
\pgfsetdash{}{0pt}
\pgfsetbuttcap

{\draw (-15\du,0\du)--(-7\du,0\du);}
{\draw (-5\du,0\du)--(3\du,0\du);}
{\draw (4.65\du,0.7\du)--(13.35\du,0.7\du);}
{\draw (4.65\du,-0.7\du)--(13.35\du,-0.7\du);}


{\pgfsetcornersarced{\pgfpoint{0.300000\du}{0.300000\du}}
\pgfsetstrokecolor{dialinecolor}
\draw (7\du,-1.2\du)--(10.8\du,0\du)--(7\du,1.2\du);}

\node[anchor=west] at (-26\du,0\du){\Large $\cU_{F_t}$
};

\node[anchor=south] at (-16\du,1.1\du){$\large \jmath$};

\node[anchor=south] at (-6\du,1.1\du){$\large  \jmath_t$};



\end{tikzpicture} &&&&\ifx\du\undefined
  \newlength{\du}
\fi
\setlength{\du}{3.3\unitlength}
\begin{tikzpicture}[transform canvas={scale=0.75}]
\pgftransformxscale{1.000000}
\pgftransformyscale{1.000000}

\definecolor{dialinecolor}{rgb}{0.000000, 0.000000, 0.000000} 
\pgfsetstrokecolor{dialinecolor}
\definecolor{dialinecolor}{rgb}{0.000000, 0.000000, 0.000000} 
\pgfsetfillcolor{dialinecolor}


\pgfsetlinewidth{0.300000\du}
\pgfsetdash{}{0pt}
\pgfsetdash{}{0pt}

\pgfpathellipse{\pgfpoint{-16\du}{0\du}}{\pgfpoint{1\du}{0\du}}{\pgfpoint{0\du}{1\du}}
\pgfusepath{stroke}
\node at (-16\du,0\du){};
\pgfpathellipse{\pgfpoint{-16\du}{0\du}}{\pgfpoint{1\du}{0\du}}{\pgfpoint{0\du}{1\du}}
\pgfusepath{fill}
\node at (-16\du,0\du){};

\pgfpathellipse{\pgfpoint{-6\du}{0\du}}{\pgfpoint{1\du}{0\du}}{\pgfpoint{0\du}{1\du}}
\pgfusepath{stroke}
\node at (-6\du,0\du){};

\pgfpathellipse{\pgfpoint{4\du}{0\du}}{\pgfpoint{1\du}{0\du}}{\pgfpoint{0\du}{1\du}}
\pgfusepath{stroke}
\node at (4\du,0\du){};

\pgfpathellipse{\pgfpoint{14\du}{0\du}}{\pgfpoint{1\du}{0\du}}{\pgfpoint{0\du}{1\du}}
\pgfusepath{stroke}
\node at (14\du,0\du){};

\pgfsetlinewidth{0.300000\du}
\pgfsetdash{}{0pt}
\pgfsetdash{}{0pt}
\pgfsetbuttcap

{\draw (-15\du,0\du)--(-7\du,0\du);}
{\draw (-5\du,0\du)--(3\du,0\du);}
{\draw (4.65\du,0.7\du)--(13.35\du,0.7\du);}
{\draw (4.65\du,-0.7\du)--(13.35\du,-0.7\du);}


{\pgfsetcornersarced{\pgfpoint{0.300000\du}{0.300000\du}}
\pgfsetstrokecolor{dialinecolor}
\draw (7\du,-1.2\du)--(10.8\du,0\du)--(7\du,1.2\du);}

\node[anchor=west] at (-26\du,0\du){\Large $F_t$
};

\node[anchor=south] at (-16\du,1.1\du){$\large \jmath$};

\node[anchor=south] at (-6\du,1.1\du){$\large \jmath_t$};



\end{tikzpicture} &\\
&&&&&\\
&\ifx\du\undefined
  \newlength{\du}
\fi
\setlength{\du}{3.3\unitlength}
\begin{tikzpicture}[transform canvas={scale=0.8}]
\pgftransformxscale{1.000000}
\pgftransformyscale{1.000000}

\definecolor{dialinecolor}{rgb}{0.000000, 0.000000, 0.000000} 
\pgfsetstrokecolor{dialinecolor}
\definecolor{dialinecolor}{rgb}{0.000000, 0.000000, 0.000000} 
\pgfsetfillcolor{dialinecolor}


\pgfsetlinewidth{0.300000\du}
\pgfsetdash{}{0pt}
\pgfsetdash{}{0pt}

\pgfpathellipse{\pgfpoint{-26\du}{0\du}}{\pgfpoint{1\du}{0\du}}{\pgfpoint{0\du}{1\du}}
\pgfusepath{stroke}
\node at (-26\du,0\du){};
\pgfpathellipse{\pgfpoint{-26\du}{0\du}}{\pgfpoint{1\du}{0\du}}{\pgfpoint{0\du}{1\du}}
\pgfusepath{fill}
\node at (-26\du,0\du){};

\pgfpathellipse{\pgfpoint{-16\du}{0\du}}{\pgfpoint{1\du}{0\du}}{\pgfpoint{0\du}{1\du}}
\pgfusepath{stroke}
\node at (-16\du,0\du){};

\pgfpathellipse{\pgfpoint{-6\du}{0\du}}{\pgfpoint{1\du}{0\du}}{\pgfpoint{0\du}{1\du}}
\pgfusepath{stroke}
\node at (-6\du,0\du){};
\pgfpathellipse{\pgfpoint{-6\du}{0\du}}{\pgfpoint{1\du}{0\du}}{\pgfpoint{0\du}{1\du}}
\pgfusepath{fill}
\node at (-6\du,0\du){};

\pgfpathellipse{\pgfpoint{4\du}{0\du}}{\pgfpoint{1\du}{0\du}}{\pgfpoint{0\du}{1\du}}
\pgfusepath{stroke}
\node at (4\du,0\du){};

\pgfpathellipse{\pgfpoint{14\du}{0\du}}{\pgfpoint{1\du}{0\du}}{\pgfpoint{0\du}{1\du}}
\pgfusepath{stroke}
\node at (14\du,0\du){};

\pgfsetlinewidth{0.300000\du}
\pgfsetdash{}{0pt}
\pgfsetdash{}{0pt}
\pgfsetbuttcap

{\draw (-25\du,0\du)--(-17\du,0\du);}
{\draw (-15\du,0\du)--(-7\du,0\du);}
{\draw (-5\du,0\du)--(3\du,0\du);}
{\draw (4.65\du,0.7\du)--(13.35\du,0.7\du);}
{\draw (4.65\du,-0.7\du)--(13.35\du,-0.7\du);}


{\pgfsetcornersarced{\pgfpoint{0.300000\du}{0.300000\du}}
\pgfsetstrokecolor{dialinecolor}
\draw (7\du,-1.2\du)--(10.8\du,0\du)--(7\du,1.2\du);}

\node[anchor=west] at (-34\du,0\du){\Large $\cM$};


\node[anchor=south] at (-16\du,1.1\du){$\large \jmath$};

\node[anchor=south] at (-6\du,1.1\du){$\large \jmath_t$};



\end{tikzpicture} &\ar!<1ex,-1ex>;[rr]!<-9ex,-1ex>&&&\ifx\du\undefined
  \newlength{\du}
\fi
\setlength{\du}{3.3\unitlength}
\begin{tikzpicture}[transform canvas={scale=0.8}]
\pgftransformxscale{1.000000}
\pgftransformyscale{1.000000}

\definecolor{dialinecolor}{rgb}{0.000000, 0.000000, 0.000000} 
\pgfsetstrokecolor{dialinecolor}
\definecolor{dialinecolor}{rgb}{0.000000, 0.000000, 0.000000} 
\pgfsetfillcolor{dialinecolor}


\pgfsetlinewidth{0.300000\du}
\pgfsetdash{}{0pt}
\pgfsetdash{}{0pt}

\pgfpathellipse{\pgfpoint{-26\du}{0\du}}{\pgfpoint{1\du}{0\du}}{\pgfpoint{0\du}{1\du}}
\pgfusepath{stroke}
\node at (-26\du,0\du){};
\pgfpathellipse{\pgfpoint{-26\du}{0\du}}{\pgfpoint{1\du}{0\du}}{\pgfpoint{0\du}{1\du}}
\pgfusepath{fill}
\node at (-26\du,0\du){};

\pgfpathellipse{\pgfpoint{-16\du}{0\du}}{\pgfpoint{1\du}{0\du}}{\pgfpoint{0\du}{1\du}}
\pgfusepath{stroke}
\node at (-16\du,0\du){};

\pgfpathellipse{\pgfpoint{-6\du}{0\du}}{\pgfpoint{1\du}{0\du}}{\pgfpoint{0\du}{1\du}}
\pgfusepath{stroke}
\node at (-6\du,0\du){};

\pgfpathellipse{\pgfpoint{4\du}{0\du}}{\pgfpoint{1\du}{0\du}}{\pgfpoint{0\du}{1\du}}
\pgfusepath{stroke}
\node at (4\du,0\du){};

\pgfpathellipse{\pgfpoint{14\du}{0\du}}{\pgfpoint{1\du}{0\du}}{\pgfpoint{0\du}{1\du}}
\pgfusepath{stroke}
\node at (14\du,0\du){};

\pgfsetlinewidth{0.300000\du}
\pgfsetdash{}{0pt}
\pgfsetdash{}{0pt}
\pgfsetbuttcap

{\draw (-25\du,0\du)--(-17\du,0\du);}
{\draw (-15\du,0\du)--(-7\du,0\du);}
{\draw (-5\du,0\du)--(3\du,0\du);}
{\draw (4.65\du,0.7\du)--(13.35\du,0.7\du);}
{\draw (4.65\du,-0.7\du)--(13.35\du,-0.7\du);}


{\pgfsetcornersarced{\pgfpoint{0.300000\du}{0.300000\du}}
\pgfsetstrokecolor{dialinecolor}
\draw (7\du,-1.2\du)--(10.8\du,0\du)--(7\du,1.2\du);}

\node[anchor=west] at (18\du,0\du){\Large$\cU'_t$};


\node[anchor=south] at (-16\du,1.1\du){$\large \jmath$};

\node[anchor=south] at (-6\du,1.1\du){$\large \jmath_t$};



\end{tikzpicture} \\
&&&\ifx\du\undefined
  \newlength{\du}
\fi
\setlength{\du}{3.3\unitlength}
\begin{tikzpicture}[transform canvas={scale=0.75}]
\pgftransformxscale{1.000000}
\pgftransformyscale{1.000000}

\definecolor{dialinecolor}{rgb}{0.000000, 0.000000, 0.000000} 
\pgfsetstrokecolor{dialinecolor}
\definecolor{dialinecolor}{rgb}{0.000000, 0.000000, 0.000000} 
\pgfsetfillcolor{dialinecolor}


\pgfsetlinewidth{0.300000\du}
\pgfsetdash{}{0pt}
\pgfsetdash{}{0pt}

\pgfpathellipse{\pgfpoint{-16\du}{0\du}}{\pgfpoint{1\du}{0\du}}{\pgfpoint{0\du}{1\du}}
\pgfusepath{stroke}
\node at (-16\du,0\du){};

\pgfpathellipse{\pgfpoint{-6\du}{0\du}}{\pgfpoint{1\du}{0\du}}{\pgfpoint{0\du}{1\du}}
\pgfusepath{stroke}
\node at (-6\du,0\du){};
\pgfpathellipse{\pgfpoint{-6\du}{0\du}}{\pgfpoint{1\du}{0\du}}{\pgfpoint{0\du}{1\du}}
\pgfusepath{fill}
\node at (-6\du,0\du){};

\pgfpathellipse{\pgfpoint{4\du}{0\du}}{\pgfpoint{1\du}{0\du}}{\pgfpoint{0\du}{1\du}}
\pgfusepath{stroke}
\node at (4\du,0\du){};

\pgfpathellipse{\pgfpoint{14\du}{0\du}}{\pgfpoint{1\du}{0\du}}{\pgfpoint{0\du}{1\du}}
\pgfusepath{stroke}
\node at (14\du,0\du){};

\pgfsetlinewidth{0.300000\du}
\pgfsetdash{}{0pt}
\pgfsetdash{}{0pt}
\pgfsetbuttcap

{\draw (-15\du,0\du)--(-7\du,0\du);}
{\draw (-5\du,0\du)--(3\du,0\du);}
{\draw (4.65\du,0.7\du)--(13.35\du,0.7\du);}
{\draw (4.65\du,-0.7\du)--(13.35\du,-0.7\du);}


{\pgfsetcornersarced{\pgfpoint{0.300000\du}{0.300000\du}}
\pgfsetstrokecolor{dialinecolor}
\draw (7\du,-1.2\du)--(10.8\du,0\du)--(7\du,1.2\du);}

\node[anchor=west] at (18\du,0\du){\Large $\cM_{F_t}$
};

\node[anchor=south] at (-16\du,1.1\du){$\large \jmath$};

\node[anchor=south] at (-6\du,1.1\du){$\large \jmath_t$};



\end{tikzpicture} \quad&&}
$$
\end{proof}

\section{Some remarks on Splitting conjecture $\bm{(\r)}$}

In this section we briefly discuss the Splitting conjecture $(\r)$. Let us start with the  following definition.

\begin{definition}
Let $X$ be a rational homogeneous manifold of Picard number one. We define:
$$
\r(X):=\left\{\begin{array}{l}\vspace{0.2cm}\min\{\r(\pi)|\,\,\pi\mbox{ universal bundle on }X\},\mbox{ if }X\neq \P^1\\
+\infty,\mbox{ if }X= \P^1.
\end{array}\right.
$$
\end{definition}

\begin{remark}\label{rem:stronger}
The fact that the Coxeter number of a connected Dynkin diagram is always bigger than the number of its nodes (see Table \ref{tab:fundeg}) implies that $\r(X)< \Cox(X)$, for $X$ rational homogeneous of Picard number one, different from $\P^1$. However, this does not imply the Conjecture $(\r)$ to be stronger than the Conjecture $(\Cox)$. In fact the logical implications between the two conjectures depend on $X$ and on the fiber $G/B$ of the bundle $\pi$ --see Table (\ref{tab:impli}) for some examples--.

\renewcommand*{\arraystretch}{1.3}
\begin{table}[h]\centering
\begin{tabular}{|c|c|c|c|c|c|c|c|c|}\cline{3-9}
\multicolumn{2}{c|}{\multirow{2}{*}{}}&\multicolumn{7}{c|}{$X$}\\
\cline{3-9}
\multicolumn{2}{c|}{}&$\DA_m(1)$&$\DB_m(1)$&$\DB_m(m)$&$\DC_m(1)$&$\DC_m(m)$&$\DD_m(1)$&$\DD_m(m)$\\
\hline
\multirow{4}{*}{$\cD$}&$\DA_n$&$(\r)\LRa(\Cox)$&$(\Cox)$&$(\r)\LRa(\Cox)$&$(\Cox)$&$(\r)\LRa(\Cox)$&$(\Cox)$&$(\r)\LRa(\Cox)$\\
\cline{2-9}
&$\DB_n$&$(\r)$&$(\r)\LRa(\Cox)$&$(\r)$&$(\r)\LRa(\Cox)$&$(\r)$&$(\r)$&$(\r)$\\
\cline{2-9}
&$\DC_n$&$(\r)$&$(\r)\LRa(\Cox)$&$(\r)$&$(\r)\LRa(\Cox)$&$(\r)$&$(\r)$&$(\r)$\\
\cline{2-9}
&$\DD_n$&$(\r)$&$(\Cox)$&$(\r)$&$(\Cox)$&$(\r)$&$(\r)\LRa(\Cox)$&$(\r)$\\
\hline
\end{tabular} 
\caption{Which conjecture ($(\Cox)$ or $(\r)$) is stronger? }
\label{tab:impli}
\end{table}

\end{remark}
\noindent The next result shows that the Conjecture $(\r)$ holds for flag bundles of type $\DA$. 

\begin{proposition}\label{prop:evidenceAk}
Let $X$ be a rational homogenous manifold of Picard number one, quotient of a simple algebraic  group $\ol G$ and $\cM$ its family of $\ol G$-isotropic lines. Let $G$ be a simple algebraic group of type $\DA_n$. If $\pi: Y \to X$ is a $G/B$-bundle, uniform with respect to $\cM$ and $\r(X)>\r(\pi)$, then $\pi: Y \to X$ is diagonalizable.
\end{proposition} 

\begin{proof}
If the Dynkin diagram of $G$ is of type $\DA_n$ then it holds that $\Cox (\pi)=\r(\pi)+1$.  If we impose $\r(\pi)<\r(X)$, then we have $\Cox(\pi)< \Cox(X)$ (since $\r(X)< \Cox(X)$ holds for every $X$, see Remark \ref{rem:stronger}). Hence, if moreover $\ol G$ is a classical group, the statement follows from Theorem \ref{thm:main}. 

We are left with the case in which $X$ is a quotient of a group $\ol G$ of exceptional type. Moreover, arguing verbatim as in the second part of the proof of Theorem \ref{thm:main}, it is enough to prove the statement in the case in which $X$ is defined by the marking of an extremal node $\jmath$ on a Dynkin diagram $\cD$ of exceptional type, that is, in the cases collected in Table \ref{tab:Ak} below.

\renewcommand*{\arraystretch}{1.3}
\begin{table}[h]\centering
\begin{tabular}{|c|c|c|c|}\hline
$X$   &$\r(X)$&  $ \cM_x$& $\dim \cM_x$ \\ 
\hline
$\DE_m(1)$ &$m-1$ &$\DD_{m-1}(m-1)$ &$(m-1)(m-2)/2$   \\
\hline
{$\DE_m(2)$} &$m-1$ &$\DA_{m-1}(3)$ &$3(m-3)$   \\
\hline
$\DE_7(7)$ &$6$ &$\DE_{6}(6)$ &$16$ \\
\hline
$\DE_8(8)$ &$7$ &$\DE_{7}(7)$ &$27$ \\
\hline
$\DF_4(1)$ &$3$& $\DC_{3}(3)$ & $6$ \\
\hline
{$\DF_4(4)$} &$3$& $\DB_3(3) $ & 6\\
\hline
\end{tabular} 
\caption{Special groups}
\label{tab:Ak}
\end{table}

Once again, by Proposition \ref{prop:diagcrit}, it is enough to prove the constancy of maps from $\cM_x$ to $\DA_n(r)$. In any of the cases, one can check that, for $n<\r(X)$:
$$\dim(\cM_x)>\left(\dfrac{\r(X)}{2}\right)^2\geq\left(\dfrac{n+1}{2}\right)^2\geq (n-r+1)r=\dim(\DA_n(r)).$$ Since $\cM_x$ has Picard number equal to one, this concludes the proof.
\end{proof}

\begin{remark}
The above statement is sharp in the case in which $X=\cD(\jmath)$ such that the smallest connected component of the Dynkin diagram $\cD\setminus\{\jmath\}$ is of type $\DA_n$. In fact, in this case, the corresponding universal bundle has rank $\r(X)$. Conjecturally, these should be the only non diagonalizable uniform bundles of the lowest possible rank. This is known to be true in the cases $X=\DA_m(r)$ (see \cite[Theorem 4.1]{MOS1}), and $X=\DD_m(m)$ (see \cite[Proposition 4.5]{MOS1}).\end{remark}

\begin{remark}
A complete proof of Conjecture $(\r)$ cannot be based on cohomological conditions on maps $\cM_x\to \cD(r)$, as the ones we have used here. For instance, there exist obvious nonconstant maps from $\DA_{m}(1)=\P^m$ to $\DC_n(1)=\P^{2n-1}$, for $n<m\leq 2n-1$. The example below (which is a generalization of \cite[Example~ 1]{Tango}) shows that we have a similar situation in the cases in which $G$ is of type $\DB_n$ or $\DD_n$. We expect that a global view on the map $s_I:\cU\to Y_I$ should provide stronger diagonalizability conditions on uniform flag bundles.
\end{remark}

\begin{example}\label{ex:mapstoquadrics} Consider the second Veronese embedding $v_2(\mathbb{P}^n) \subset \mathbb{P}(S^2\C^{n+1})$.
Given homogeneous coordinates $(x_0:\dots:x_n)$ in $\P(\C^{n+1})$, the polynomials $Z_{i,j}:=x_ix_j$  provide a set of homogeneous coordinates in $\P(S^2\C^{n+1})$, and a set of generators of the ideal of $v_2(\mathbb{P}^n)$ of the form $Z_{i,j}Z_{k,l}-Z_{i,l}Z_{k,j}$. 

Set, for $n$ odd and even respectively:
$$
F_{odd}:=\sum_{i=0}^{(n-1)/2} (Z_{2i,2i}Z_{2i+1,2i+1}-Z_{2i,2i+1}^2)=0 $$
$$F_{even}:=\sum_{i=0}^{(n-2)/2} (Z_{2i,2i}Z_{2i+1,2i+1}-Z_{2i,2i+1}^2)+(Z_{n-2,n-1}Z_{n,n}-Z_{n-2,n}Z_{n,n-1})=0
$$
The vertex of $F_{odd}$ (resp. $F_{even}$) has codimension $3(n+1)/2$ (resp. $3(n+2)/2$), being defined respectively by the linear equations:
$$Z_{j,j}=Z_{2i,2i+1}=0,\quad 0 \leq j,2i+1 \leq n,$$
$$Z_{j,j}=Z_{2i,2i+1}=Z_{n-1,n}=Z_{n-2,n}=0,\quad 0 \leq j,2i+1 \leq n.$$
Since, in each case, the vertex $V$ of the quadric is contained in the subset given by the equations $Z_{j,j}=0$, it follows that $V\cap v_2(\mathbb{P}^n)=\emptyset$. Hence, the linear projection of $v_2(\mathbb{P}^n)$ from $V$ is contained in a smooth quadric, of dimension $(3n-1)/2$, for $n$ odd, or $(3n+2)/2$, for $n$ even. Summing up, we have constructed non constant morphisms from $\mathbb{P}^n$ to $\mathbb{Q}^{3\lfloor \frac{n}{2} \rfloor+1}$, for every $n$.
\end{example}

\begin{problem}
It is known, see \cite[Theorem 4.1]{L3}, that  there are no non constant maps from $\mathbb{P}^{n}$ to $\mathbb{Q}^{n}$. Find an optimal bound $b(n)$ such that the only morphisms $f:\mathbb{P}^n \to  \mathbb{Q}^{m}$ with $m \leq b(n)$ are the constant maps.
\end{problem}


\bibliographystyle{plain}
\bibliography{biblio}

\begin{thebibliography}{10}

\bibitem{Ballico2}
Edoardo Ballico.
\newblock Uniform vector bundles on quadrics.
\newblock {\em Ann. Univ. Ferrara Sez. VII (N.S.)}, 27:135--146 (1982), 1981.

\bibitem{Ballico1}
Edoardo Ballico.
\newblock Uniform vector bundles of rank {$(n+1)$} on {${\bf P}_{n}$}.
\newblock {\em Tsukuba J. Math.}, 7(2):215--226, 1983.

\bibitem{Bo}
Armand Borel.
\newblock Sur la cohomologie des espaces fibr\'es principaux et des espaces
  homog\`enes de groupes de {L}ie compacts.
\newblock {\em Ann. of Math. (2)}, 57:115--207, 1953.

\bibitem{Drez80}
Jean-Marc Drezet.
\newblock Exemples de fibres uniformes non homog\`enes sur {${\bf P}_{n-}$}.
\newblock {\em C. R. Acad. Sci. Paris S\'er. A-B}, 291(2):A125--A128, 1980.

\bibitem{3264}
David Eisenbud and Joe Harris.
\newblock {\em 3264 and all that---a second course in algebraic geometry}.
\newblock Cambridge University Press, Cambridge, 2016.

\bibitem{Elen}
G.~Elencwajg.
\newblock Des fibr\'es uniformes non homog\`enes.
\newblock {\em Math. Ann.}, 239(2):185--192, 1979.

\bibitem{EHS}
G.~Elencwajg, A.~Hirschowitz, and M.~Schneider.
\newblock Les fibr\'es uniformes de rang au plus {$n$}\ sur {${\bf P}_{n}({\bf
  C})$}\ sont ceux qu'on croit.
\newblock In {\em Vector bundles and differential equations ({P}roc. {C}onf.,
  {N}ice, 1979)}, volume~7 of {\em Progr. Math.}, pages 37--63. Birkh\"auser,
  Boston, Mass., 1980.

\bibitem{Ellia}
Philippe Ellia.
\newblock Sur les fibr\'es uniformes de rang {$(n+1)$} sur {${\bf P}^{n}$}.
\newblock {\em M\'em. Soc. Math. France (N.S.)}, 7:1--60, 1982.

\bibitem{Fult92}
William Fulton.
\newblock Flags, {S}chubert polynomials, degeneracy loci, and determinantal
  formulas.
\newblock {\em Duke Math. J.}, 65(3):381--420, 1992.

\bibitem{Fult}
William Fulton.
\newblock {\em Intersection theory}, volume~2 of {\em Ergebnisse der Mathematik
  und ihrer Grenzgebiete. 3. Folge. A Series of Modern Surveys in Mathematics
  [Results in Mathematics and Related Areas. 3rd Series. A Series of Modern
  Surveys in Mathematics]}.
\newblock Springer-Verlag, Berlin, second edition, 1998.

\bibitem{FMSS}
William Fulton, Robert~D. MacPherson, Frank Sottile, and Bernd Sturmfels.
\newblock Intersection theory on spherical varieties.
\newblock {\em J. Algebraic Geom.}, 4(1):181--193, 1995.

\bibitem{Gro1}
Alexander Grothendieck.
\newblock Sur la classification des fibr\'es holomorphes sur la sph\'ere de
  {R}iemann.
\newblock {\em Amer. J. Math.}, 79:121--138, 1957.

\bibitem{Guyot}
Michelle Guyot.
\newblock Caract\'erisation par l'uniformit\'e des fibr\'es universels sur la
  grassmanienne.
\newblock {\em Math. Ann.}, 270(1):47--62, 1985.

\bibitem{Hi82}
Howard Hiller.
\newblock {\em Geometry of {C}oxeter groups}, volume~54 of {\em Research Notes
  in Mathematics}.
\newblock Pitman (Advanced Publishing Program), Boston, Mass.-London, 1982.

\bibitem{Hum2}
James~E. Humphreys.
\newblock {\em Introduction to {L}ie algebras and representation theory},
  volume~9 of {\em Graduate Texts in Mathematics}.
\newblock Springer-Verlag, New York, 1978.
\newblock Second printing, revised.

\bibitem{Hum3}
James~E. Humphreys.
\newblock {\em Reflection groups and {C}oxeter groups}, volume~29 of {\em
  Cambridge Studies in Advanced Mathematics}.
\newblock Cambridge University Press, Cambridge, 1990.

\bibitem{LM}
Joseph~M. Landsberg and Laurent Manivel.
\newblock On the projective geometry of rational homogeneous varieties.
\newblock {\em Comment. Math. Helv.}, 78(1):65--100, 2003.

\bibitem{L3}
Robert Lazarsfeld.
\newblock Some applications of the theory of positive vector bundles.
\newblock In {\em Complete intersections ({A}cireale, 1983)}, volume 1092 of
  {\em Lecture Notes in Math.}, pages 29--61. Springer, Berlin, 1984.

\bibitem{MOS22}
Roberto Mu{\~n}oz, Gianluca Occhetta, and Luis~E. Sol{\'a}~Conde.
\newblock Rank two {F}ano bundles on {$\Bbb{G}(1,4)$}.
\newblock {\em J. Pure Appl. Algebra}, 216(10):2269--2273, 2012.

\bibitem{MOS1}
Roberto Mu{\~n}oz, Gianluca Occhetta, and Luis~E. Sol{\'a}~Conde.
\newblock Uniform vector bundles on {F}ano manifolds and applications.
\newblock {\em J. Reine Angew. Math.}, 664:141--162, 2012.

\bibitem{MOS5}
Roberto Mu{\~n}oz, Gianluca Occhetta, and Luis~E. Sol{\'a}~Conde.
\newblock On uniform flag bundles on {F}ano manifolds.
\newblock Preprint arXiv:{\tt 1610.05930}, 2016.

\bibitem{MOSW}
Roberto Mu{\~n}oz, Gianluca Occhetta, Luis~E. Sol{\'a}~Conde, and Kiwamu
  Watanabe.
\newblock Rational curves, {D}ynkin diagrams and {F}ano manifolds with nef
  tangent bundle.
\newblock {\em Math. Ann.}, 361(3):583--609, 2015.

\bibitem{OSW}
Gianluca Occhetta, Luis~E. Sol\'a~Conde, and Jaros\l aw~A. Wi\'sniewski.
\newblock Flag bundles on {F}ano manifolds.
\newblock {\em J. Math. Pures Appl. (9)}, 106(4):651--669, 2016.

\bibitem{OSS}
Christian Okonek, Michael Schneider, and Heinz Spindler.
\newblock {\em Vector bundles on complex projective spaces}.
\newblock Progress in Mathematics, 3. Birkh\"auser, Boston, Mass., 1980.

\bibitem{Pan}
Xuanyu Pan.
\newblock Triviality and split of vector bundles on rationally connected
  varieties.
\newblock {\em Math. Res. Lett.}, 22(2):529--547, 2015.

\bibitem{Sa}
Ei-ichi Sato.
\newblock Uniform vector bundles on a projective space.
\newblock {\em J. Math. Soc. Japan}, 28(1):123--132, 1976.

\bibitem{Tango}
Hiroshi Tango.
\newblock On morphisms from projective space {$P^{n}$} to the {G}rassmann
  variety {$G{\rm r}(n, d)$}.
\newblock {\em J. Math. Kyoto Univ.}, 16(1):201--207, 1976.

\bibitem{VdV}
Antonius {Van de Ven}.
\newblock {On uniform vector bundles.}
\newblock {\em {Math. Ann.}}, 195:245--248, 1972.

\bibitem{Wis2}
Jaros{\l}aw~A. Wi{\'s}niewski.
\newblock Uniform vector bundles on {F}ano manifolds and an algebraic proof of
  {H}wang-{M}ok characterization of {G}rassmannians.
\newblock In {\em Complex geometry ({G}\"ottingen, 2000)}, pages 329--340.
  Springer, Berlin, 2002.

\end{thebibliography}

\end{document}